\documentclass[10pt,reqno]{amsart}
\usepackage{amssymb,mathrsfs,color}
\usepackage{pinlabel}

\usepackage{enumerate}

\usepackage{graphicx}
\usepackage{xcolor} 
\usepackage{tensor}
\usepackage{slashed}
\usepackage{cite}
\definecolor{green}{rgb}{0,0.5,0} 

\newcommand{\set}[1]{\{#1\}}

\newcommand{\tr}{\textrm{tr}}

\newcommand{\ud}{\mathrm{d}}
\newcommand{\rd}{\partial}

\newcommand{\bb}{\Big}

\newcommand{\eps}{\epsilon}

\newcommand{\lmb}{\lambda}
\newcommand{\Lmb}{\Lambda}
\newcommand{\sgm}{\sigma}

\newcommand{\Tht}{\Theta}

\newcommand{\omg}{\omega}

\newcommand{\bbD}{\mathbb D}

\newcommand{\bbH}{\mathbb H}

\newcommand{\bbR}{\mathbb R}
\newcommand{\bbS}{\mathbb S}

\newcommand{\calD}{\mathcal D}

\newcommand{\calL}{\mathcal L}

\newcommand{\arctanh}{\mathrm{arctanh} \, }

\newcommand{\gapE}{\mu^{2}_{\lmb,k}}
\newcommand{\rnL}{\widetilde{\calL}}

\newcommand{\ttmatrix}[4]{\left( \begin{array}{cc} #1 & #2 \\ #3 & #4 \end{array} \right)}

\usepackage{wrapfig}
\usepackage{tikz}
\usetikzlibrary{arrows,calc,decorations.pathreplacing}
\definecolor{light-gray1}{gray}{0.90}
\definecolor{light-gray2}{gray}{0.80}

\definecolor{deepgreen}{cmyk}{1,0,1,0.5}

\newcommand{\E}{\mathcal{E}}

\newcommand{\HH}{\mathcal{H}}

\newcommand{\LL}{\mathcal{L}}
\newcommand{\NN}{\mathcal{N}}
\newcommand{\M}{\mathcal{M}}

\newcommand{\EE}{\mathscr{E}}

\newcommand{\C}{\mathbb{C}}
\newcommand{\Hp}{\mathbb{H}}
\newcommand{\N}{\mathbb{N}}
\newcommand{\R}{\mathbb{R}}
\newcommand{\Sp}{\mathbb{S}}
\newcommand{\Z}{\mathbb{Z}}
\newcommand{\D}{\mathbb{D}}

\newcommand{\h}{\mathbf{h}}

\newcommand{\al}{\alpha}
\newcommand{\be}{\beta}

\newcommand{\de}{\delta}
\newcommand{\e}{\varepsilon}
\newcommand{\fy}{\varphi}
\newcommand{\om}{\omega}
\newcommand{\la}{\lambda}
\newcommand{\te}{\theta}
\newcommand{\s}{\sigma}

\newcommand{\z}{\zeta}

\newcommand{\La}{\Lambda}

\newcommand{\Te}{\Theta}

\newcommand{\p}{\partial}

\newcommand{\loc}{\operatorname{loc}}

\makeatletter

\newcommand{\Rmnum}[1]{\expandafter\@slowromancap\romannumeral #1@}
\makeatother

\newcommand{\ti}{\widetilde}
\newcommand{\ba}{\overline}

\newcommand{\abs}[1]{\left\lvert{#1}\right\rvert}

\newcommand{\ali}[1]{\begin{align}\begin{split} #1 \end{split}\end{align}}
\newcommand{\ant}[1]{\begin{align*}\begin{split} #1 \end{split}\end{align*}}
\newcommand{\EQ}[1]{\begin{equation}\begin{split} #1 \end{split}\end{equation}}

\newcommand{\Del}[1]{}

\numberwithin{equation}{section}

\newtheorem{thm}{Theorem}[section]

\newtheorem{lem}[thm]{Lemma}
\newtheorem{prop}[thm]{Proposition}

\newtheorem{claim}[thm]{Claim}
\theoremstyle{remark}
\newtheorem{rem}{Remark}
\newtheorem{defn}{Definition}

\newcommand{\mand}{{\ \ \text{and} \ \  }}

\newcommand{\mif}{{\ \ \text{if} \ \ }}
\newcommand{\mfor}{{\ \ \text{for} \ \ }}
\newcommand{\mas}{{\ \ \text{as} \ \ }}

\newcommand{\euc}{\textrm{euc}}
\newcommand{\Lline}{\LL_{V_{\la,k}}}
\newcommand{\Qu}{Q_{\la,k}}
\newcommand{\Lv}{\ti{\LL}_{V_{\la,k}}}

\newcommand{\dvol}{\operatorname{dVol}} 

\newcommand{\VV}{\mathcal{V}}
\newcommand{\WW}{\mathcal{W}}

\newcommand{\atan}{\mathrm{arctan}}
\newcommand{\YM}{\mathrm{YM}}

\newcommand{\rrho}{\rho}
\newcommand{\ssigma}{\sigma}
\newcommand{\ttau}{\tau}

\begin{document}

\title[Gap eigenvalues for geometric equations]{Gap eigenvalues and asymptotic dynamics of geometric wave equations on hyperbolic space}

\author{Andrew Lawrie}
\author{Sung-Jin Oh}
\author{Sohrab Shahshahani}

\begin{abstract} 
In this paper we study  $k$-equivariant wave maps from the hyperbolic plane into the $2$-sphere as well as the energy critical equivariant $SU(2)$ Yang-Mills problem on $4$-dimensional hyperbolic space. The latter problem bears many similarities to a $2$-equivariant wave map into a surface of revolution. As in the case of $1$-equivariant wave maps considered in~\cite{LOS1}, both problems admit a family of stationary solutions  indexed by a parameter that determines how far the image of the map wraps around the target manifold.    
Here we show that if the image of a stationary solution is contained in a geodesically convex subset of the target, then it  is asymptotically stable in the energy space. However,  for a stationary solution that  covers a large enough portion of the target, we prove that the Schr\"odinger operator obtained by linearizing about such a harmonic map admits a simple positive eigenvalue in the spectral gap.  
As there is no a priori nonlinear obstruction to asymptotic stability, this gives evidence for the existence of metastable states (i.e., solutions with anomalously slow decay rates) in these simple geometric models.

\end{abstract}

\thanks{Support of the National Science Foundation,  DMS-1302782 and NSF 1045119 for the first and third authors, respectively, is gratefully acknowledged. The second author is a Miller Research Fellow, and acknowledges support from the Miller Institute.}
\maketitle

\section{Introduction}\label{intro}

We consider the $k$-equivariant wave maps equation from $\R \times \Hp^2 \to \Sp^2$ and the energy critical equivariant Yang-Mills problem on $\R \times \Hp^4$ with gauge group $SU(2)$. After the usual equivariant reductions, both equations take the form 
\EQ{ \label{eq wm}
\psi_{tt} - \psi_{rr} - \coth r \, \psi_r +k^2 \frac{ g(\psi)g'(\psi)}{\sinh^2 r} = 0,
}
where $(\psi, \theta)$ are geodesic polar coordinates on a target surface  of revolution $\M$, and $g$ determines the metric, $ds^2 = d \psi^2 + g^2(\psi) d \theta^2$. In the case of $k$-equivairant wave maps, we set $\M=\Sp^2$ and  $g(\psi) = \sin \psi$. For the Yang-Mills problem, $k = 2$ and  $g( \psi) =  \psi - \frac{1}{2} \psi^2$.  
Note that one can view the latter problem as a $2$-equivariant wave map from $\R \times \Hp^2$ into a surface of revolution which is diffeomorphic to $\Sp^2$ by restricting to  $\psi\in[0,2]$. Indeed, $g(0) = g(2) = 0$, $\psi =1$ is the unique zero of $g^\prime( \psi)$ and $\{ \psi \leq1\}$ defines the largest geodesically convex neighborhood of the ``north pole," $\psi=0.$ 

 There are several features of these models that make them interesting testing grounds in the study of asymptotic dynamics of dispersive equations on curved spaces. As we will see below,  the introduction of hyperbolic geometry on the domain allows for an abundance of finite energy stationary solutions to~\eqref{eq wm} -- these are harmonic maps in the case of the wave maps equation.  The curved background also eliminates any natural scaling invariance for the problem, thus removing an a priori obstruction to the asymptotic stability of stationary solutions. However the energy-criticality of these equations is still manifest since solutions that  concentrate at very small scales can be well approximated by solutions to the underlying Euclidean problems, which are energy critical;  to obtain the scale invariant Euclidean equation simply replace $ \coth r$ by $r^{-1}$ and $\sinh^2r$ by $r^2$ in~\eqref{eq wm}.  Taking this last point a bit further, one expects that the fundamental blow-up constructions for the corresponding Euclidean models of Krieger, Schlag, Tataru~\cite{KST, KST2}, Rodnianski, Sterbenz~\cite{RS}, and Rapha\"el, Rodnianski~\cite{RR} should carry over to the present hyperbolic setting. 
 
 In this paper, we study a more subtle effect of the underlying scale invariant Euclidean equation on the asymptotic dynamics of the  hyperbolic problem~\eqref{eq wm},  namely the formation of eigenvalues in the spectral gap of the Schr\"odinger operator obtained by linearizing about a stationary solution to~\eqref{eq wm} that wraps sufficiently far around the target manifold $\M$. This phenomenon was discovered in~\cite{LOS1} in the case of $1$-equivariant wave maps from $\R\times\Hp^2$ to $\Sp^2$. Here we establish the existence of gap eigenvalues for higher equivalence classes $k \ge2$, and for the equivariant Yang-Mills problem, while further elucidating the role that the geometry of the image of the stationary  map plays in the construction.  In particular, we show that there are no gap eigenvalues associated to $k$-equivariant harmonic maps and stationary solutions to the Yang-Mills problem whose image lies in a region of the target which is slightly larger than the maximal geometrically convex subset containing the north pole -- when the target is $\Sp^2$ this is simply the northern hemisphere plus a small band in the southern hemisphere. In fact, in this case we can find a uniform-in-$k$ latitude $\al_*$ (which is slightly below the equator), under which there are no gap eigenvalues.  Moreover, we find a uniform-in-$k$ latitude $\al^*$ (which is further into the southern hemisphere than $\al_*$), with the property that the linearized operator associated to any $k$-equivaraint harmonic map whose image contains the latitude $\al^*$ must have an eigenvalue in  its spectral gap.

Note that the existence (or the non-existence) of a gap eigenvalue of the linearized operator is a property of harmonic maps, independent of the wave map dynamics. The same remark applies to the Yang-Mills problem. This property, however, has a particularly interesting implication on the \emph{dispersive} dynamics (such as the wave map or Yang-Mills dynamics), rather than on the elliptic (harmonic map or elliptic Yang-Mills) or parabolic (harmonic map or Yang-Mills heat flow) analogs. Indeed, the gap eigenvalue gives rise to a time-periodic (hence non-decaying) solution to the linearized wave equation, which rules out any \emph{linear} mechanism for asymptotic stability of the stationary solution. This situation should be compared to the case of, for example, the linearized parabolic equation, whose solutions always exhibit exponential decay in time as the gap eigenvalue is positive. Nevertheless, there are reasons to believe that these stationary solutions are still asymptotically stable under the wave map (or Yang-Mills) dynamics, possibly by a \emph{nonlinear} mechanism called ``radiative damping''. We refer to Remark~\ref{rem:FGR} for further discussion.

In order to describe our main results in more detail we first give a more precise account of the setup. 
Let $(r, \theta)$ be geodesic  polar coordinates on the hyperboloid model of $\Hp^2$:
\ant{
[0, \infty) \times \bbS^1 \ni (r, \theta) \mapsto (\sinh r \sin \theta, \sinh r\cos \theta, \cosh r) \in \R^{2+1}.
}
Denote this map by $\Psi: [0, \infty) \times \bbS^1 \to (\R^{2+1}, \mathbf{m})$, where $\bf{m}$ is the Minkowski metric on $\R^{2+1}$. The  metric $\h$ on $\Hp^2$ in these coordinates is given by the pullback of the Minkowski metric by  $\Psi$, i.e.,  $\h = \Psi^*\mathbf{m}$ and $ \h =  \textrm{diag}(1, \sinh^2 r)$. 
The volume element is $\sqrt{\abs{\h(r, \theta)}} = \sinh r$, and thus for $f : \Hp^2 \to \R$, 
\ant{
\int_{\Hp^2} f(x) \,  \dvol_{\h} = \int_0^{2\pi} \int_0^{\infty} f(\Psi(r ,\theta)) \sinh r \, dr \, d  \theta.
}
For radial functions, $f: \Hp^2 \to \R$ we abuse notation and write $f(x) = f(r)$. 
\vspace{\baselineskip} 

\subsection{ $k$-equivariant wave maps $U: \R \times \Hp^2 \to \Sp^2$:}  We first describe our results for  wave maps $U: \R \times \Hp^2 \to \Sp^2$. Since the domain $\Hp^2$  and the target $\Sp^2$ are rotationally symmetric we can consider a restricted class of maps satisfying an equivariance assumption $U \circ \rho = \rho^k \circ U$, for all  $\rho \in SO(2)$.  
This leads to the usual $k$-equivariant anastz
\ant{
U(t, r, \theta) = (\psi(t, r), \theta) \hookrightarrow ( \sin \psi \sin k\theta, \sin \psi \cos  k\theta, \cos \psi).
}
Here $\psi$ measures the polar angle from the north pole of $\Sp^2$. 
In this formulation, $k$-equivariant wave maps are formal critical points of the Lagrangian
\ant{
\mathcal{L}(U)= \frac{1}{2} \int_{\R} \int_0^{\infty} \left( - \psi_t^2(t, r) + \psi_r^2(t, r) + k^2\frac{ \sin^2 \psi(t, r)}{\sinh^2 r} \right) \, \sinh r \, dr \, dt.
}
The Euler-Lagrange equations reduce to an equation for the angle $\psi$. Indeed we consider here the  Cauchy problem, 
\EQ{\label{wm}
&\psi_{tt} - \psi_{rr} - \coth r \, \psi_r + k^2\frac{ \sin(2 \psi)}{2 \sinh^2 r} = 0,\\
& \vec \psi(0)= ( \psi_0, \psi_1).
}
We use the notation $\vec \psi(t)$ to denote the vector $\vec \psi(t, r):= ( \psi(t, r), \psi_t(t, r))$. The conserved energy is given by
\EQ{\label{con energy}
\E( \vec \psi(t)) = \frac{1}{2} \int_0^{\infty}  \left( \psi_t^2 + \psi^2_r + k^2\frac{\sin^2 \psi}{\sinh^2 r} \right) \sinh r \, dr = \textrm{const}.
}
Note that in order for  initial data $\vec \psi(0) = (\psi_0, \psi_1)$ to have finite energy the above requires  $\psi_0(0) = \ell \pi$ for some $\ell \in \Z$. Moreover,  continuous dependence on the  initial data on a time interval $I$, dictates that the  integer $\ell$ is preserved by the evolution on $I$. We restrict to the case $\ell =0,$ corresponding to maps that send $r=0$ (the vertex of the hyperboloid) to the north pole of $\Sp^2$, as all other cases  can be obtained from this one via the change of variables $\psi \mapsto \psi+  \ell  \pi$. 

The behavior of finite energy data at $r=\infty$ is more flexible. Indeed,  $\psi_0(r)$ has a well-defined limit as $r\rightarrow\infty,$ however this limit can be any real number, i.e., $\E(\psi_0,\psi_1)<\infty$ means that  there exists $\alpha\in\R$ so that $\lim_{r\to \infty}\psi_0(r)=\alpha$.  This is in sharp contrast to the corresponding problem for Euclidean wave maps $\R^{1+2} \to \Sp^2$, where the endpoint can only be an integer multiple of $\pi$ --  such maps then have a fixed topological degree. Here, the fact that any finite endpoint is allowed can be attributed  to the rapid decay of $\sinh^{-1} r$ as $ r \to \infty$ in the last term in the integrand of~\eqref{con energy}, and is ultimately responsible for the existence of the family of finite energy harmonic maps to be described below.

Although arbitrary endpoints are allowed, here we only consider  maps so that $\psi_0( \infty) = \al$ for $\al \in [0, \pi)$, which means that we will only consider those $\psi_0$ that do not reach the south pole. For a fixed equivariance class $k \in \N$, the space of finite energy data with endpoints $\al \in [0, \pi)$ can be divided into disjoint classes, 
\EQ{
\E_{\la, k} := \{ ( \psi_0, \psi_1) \mid \E(\psi_0,\psi_1)< \infty, \, \,  \psi_0(0) =0, \, \, \psi_0( \infty) = 2 \arctan( \la^k)\}
}
for $\la \in [0, \infty)$. The reason for this restriction to endpoints $\al \in [0, \pi)$ is that in each $\E_{\la, k}$ there is a unique $k$-equivariant harmonic map $Q_{\la,k}$, i.e., a solution to 
\EQ{\label{hm}
&Q_{rr}+ \coth r\, Q_r = k^2\frac{\sin 2Q}{2\sinh^2 r},\\
&Q(0) = 0, \quad \lim_{r \to \infty} Q(r) = 2\arctan(\la^k).
}
Indeed, $Q_{\la, k}$ is given by the explicit formula, 
\EQ{
Q_{\la,k}(r) := 2\arctan(\la ^k\tanh^k(r/2)).
}
Moreover, $(Q_{\la,k}, 0)$ minimizes the energy in $\E_{\la, k}$,  
\ant{
\E( Q_{\la,k}, 0) = 1- \cos(Q_{\la,k}( \infty)) = 2k \frac{\la^{2k}}{ \la^{2k} +1}.
}
In other words, for each angle $\al \in [0, \pi)$ there exists a $k$-equivariant map connecting $0$ to $\al$ of minimum energy and this map is, in fact, the \emph{harmonic map} $Q_{\la,k}$, with $\la := \tan^{\frac{1}{k}}( \al/2)$.  We note that $\la =1$ corresponds to  $\al = \pi/2$ and thus $Q_{1, k}$ covers precisely the northern hemisphere of $\Sp^2$. For endpoints $\al \ge \pi$ there are \emph{no} finite energy harmonic maps. 

 The existence of the $Q_{\la,k}$ is in stark contrast to the corresponding Euclidean problem,  $k$-equivariant wave maps $\R^{1+2} \to \Sp^2$, which also reduces to an equation for the polar angle $\psi$:
  \EQ{\label{euc wm}
  \psi_{tt} - \psi_{rr} - \frac{1}{r} \psi_r + k^2\frac{ \sin 2 \psi}{2 r^2} = 0.
  }
  In fact, the unique (up to scaling and sign reversal) nontrivial Euclidean $k$-equivariant harmonic map is given by $Q_{\euc, k}(r) = 2 \arctan(r^k)$, which connects the north pole to the south pole of the sphere.  Indeed, $Q_{\euc, k}$  solves 
   \EQ{\label{euc hm}
 Q_{rr} + \frac{1}{r} Q_{r}  = k^2\frac{\sin 2Q}{2 r^2}, \, \, \, Q(0) = 0,
 }
and minimizes the Euclidean energy
 \EQ{
\E_{\euc} (\psi_0, \psi_1) =  \frac{1}{2} \int_0^{\infty}\left( (\p_r \psi_0)^2 + \psi_1^2 + k^2\frac{\sin^2 \psi_0}{r^2} \right)\, r \, dr
 }
amongst all $k-$equivariant maps which satisfy $\psi_0(0) = 0, \psi_0(\infty) = \pi$.  By direct computation we have $\E_{\euc}(Q_{\euc, k}, 0) = 2k$. We note that for the hyperbolic harmonic maps $Q_{\la,k}$ we have
 \ant{
&\E(Q_{\la,k}, 0) \to \E_{\euc}(Q_{\euc, k}, 0) \mas \la \to \infty, \\
&\E(Q_{\la,k}, 0) \to 0 \mas \la \to 0.
}
It is well known that  $Q_{\euc, k}$ is \emph{unstable} with respect to the Euclidean equivariant wave map flow. This instability leads to finite time blow-up, see~\cite{Struwe, Cote, KST, RS, RR}.  
\vspace{\baselineskip}

We now introduce the setup for studying the asymptotic dynamics of the wave map evolution~\eqref{wm} in the energy class $\E_{\la,k}$. The natural space in which to consider solutions to~\eqref{wm} with $\psi(t, \infty) = 0$ is the energy space
\EQ{\label{H0 def}
\| (\psi_0, \psi_1)\|_{\HH_0}^2:=  \int_0^{\infty}   \left((\p_r\psi_0)^2(r) + \psi_1^2(r) + k^2\frac{\psi_0^2(r)}{\sinh^2 r} \right) \, \sinh r \, dr.
}
To measure solutions with nontrivial endpoint $2\atan(\la^k)$, we endow $\E_{\la,k}$ with the  ``norm"
\EQ{
\| (\psi_0, \psi_1)\|_{\E_{\la,k}} := \| (\psi_0, \psi_1) -(Q_{\la,k}, 0)\|_{\HH_0}.
}
The central object of study in this paper will be the Schr\"odinger operator obtained by linearizing~\eqref{wm} about $\Qu$. 
Understanding this operator is fundamental to the study of the asymptotic dynamics of solutions near $\Qu$; see Remark~\ref{rem:stab} for a more precise discussion.
To define the linearized Schr\"odinger operator, we first pass to a radial wave equation on $\R\times\Hp^{2k+2}$ since the linear part of~\eqref{wm} provides more dispersion than a free wave on $\R \times \Hp^2$. In fact, as we will see below, for $0 \le\la^k < 1+ \de_0$ solutions to the linearized version of \eqref{wm} about $Q_{\la, k}$ enjoy the same dispersion as free waves on $\R \times \Hp^d$ with $d = 2k +2$. To see this we make a change of variables: for a solution $\vec \psi(t) \in \E_{\la,k}$ define $ u(t)$ by
\EQ{\label{u S2 def}
 \sinh^k r \,  u(t, r) :=  \psi(t, r) - Q_{\la,k}(r).
}
We obtain the following equation for $\vec u(t)$,
\EQ{\label{u eq}
&u_{tt}- u_{rr} - (2k+1) \coth r \, u_r - k(k+1) u + V_{\la,k}(r) u = \NN_{\Sp^2}(r, u)\\
&\vec u(0)= (u_0, u_1)
}
where the \emph{attractive} potential $V_{\la,k}$ is given by  
\begin{align}
& V_{\la,k}(r) := k^2\frac{ \cos 2Q_{\la,k} -1}{ \sinh^2 r} \le 0, \label{Vla}
\end{align}
and the nonlinearity $\NN_{\Sp^2}$ is 
\EQ{
&\NN_{\Sp^2}(r, u):= k^2\sin (2Q_{\la,k}) \frac{\sin^2( 2 \sinh ^kr\, u)}{ \sinh^{2+k} r}   \\
& \quad +k^2 \cos(2Q_{\la, k})\frac{2 \sinh r\, u -  \sin (2 \sinh ^kr \, u)}{2 \sinh^{2+k} r}. \label{N S}
}
In fact, one can show  that the Cauchy problem~\eqref{u eq} for data $(u_0, u_1) \in H^1  \times L^2 (\Hp^{2k+2})$, is equivalent to the Cauchy problem for the the polar angle $\vec \psi$, i.e.,~\eqref{wm}. We give a few details regarding this equivalence in Section~\ref{sec:pre}.

The underlying \emph{linear} equation is then, 
\EQ{ \label{v eq}
v_{tt} - \Delta_{\Hp^{2k+2}} v -k(k+1) v + V_{\la,k} v = 0
}
for radially symmetric functions $v(t):  \Hp^{2k+2} \to \R$, and we define self-adjoint Schr\"odinger operators
\EQ{\label{H Vla}
&H_{0,k}:=-\p_{rr}- (2k+1) \coth r \, \p_r -k(k+1),\\
&H_{V_{\la,k}} := -\p_{rr}- (2k+1) \coth r \, \p_r -k(k+1) +V_{\la,k}.
}
It is well known that the  spectrum~$ \sigma( H_{V_{\la, k}})$  plays a central role in determining the dispersive properties of~\eqref{v eq}, and thus in determining the asymptotic dynamics of $Q_{\la, k}$. We recall that the spectrum of the Laplacian on $\Hp^{d}$   is given by 
\ant{
 \s(- \Delta_{\Hp^{d}}) =  \left[ (d-1)^2/4, \infty \right)
 }
 and  therefore,  by setting  $d = 2k+2$ above we see that  for $H_{0, k}$, 
 \EQ{
 \s(H_{0, k}) = [ 1/4, \infty)
 }
The goal here will be to understand the spectrum of the perturbed operator $H_{V_{\la, k}}$. Our first result states that for harmonic maps $Q_{\la, k}$ whose image is contained in a region that is slightly larger than the northern hemisphere of the sphere, the corresponding linearized operator $H_{V_{\la, k}}$ has purely absolutely continuous spectrum equal to $[1/4, \infty)$. 
\begin{thm}\label{S2 no eval}
\begin{itemize}
\item [$(i)$] Fix any equivariance class $k \in \N$.  For each  $0 \le \lambda\leq1$,   the spectrum of $H_{V_{\la,k}}$ is purely absolutely continuous and is given by  $$\sigma(H_{V_{\la,k}})=[1/4,\infty).$$  In particular, $H_{V_{\la, k}}$ has no negative spectrum, no eigenvalues in the spectral gap $[0,1/4),$ and the threshold $1/4$ is neither a resonance nor an eigenvalue.
\item[$(ii)$] Let $\Theta = \Te(k) := \la^k$. There exists $\Te_*>1$, which is independent of $k \in \N$, so that if $\Theta <\Te_*$, then the same conclusions as in part $(i)$ hold for the spectrum of $H_{V_{\la,k}}$ with $\la = \Te^{1/k}$. In particular, there is an angle $$\al_*:= 2\arctan(\Te_*)> \pi/2$$ independent of $k$,  such that any $k$-equivariant harmonic map with  $$\lim_{r \to \infty} Q_{\la, k}(r) < \al_*$$  yields a linearized operator $H_{V_{\la, k}}$ satisfying the conclusions in part $(i)$. 
\end{itemize}
\end{thm}
\begin{rem}  \label{rem:stab}
One can use the conclusions of Theorem~\ref{S2 no eval} to prove Strichartz estimates for~\eqref{v eq} by viewing $H_{V_{\la, k}}$ as a perturbation of $H_{0,k}$. The proof relies on the \emph{distorted Fourier transform} relative to $H_{V_{\la, k}}$ and the associated Weyl-Titchmarsh theory. Indeed, following the arguments in~\cite[Section $4$]{LOS1} one can prove integrated localized energy estimates by establishing  precise decay estimates for  the spectral measure associated to $H_{V_{\la, k}}$, which in turn can be used to establish Strichartz estimates. With Strichartz estimates in hand, it is then standard to prove a small data scattering result for~\eqref{u eq}, from which the asymptotic stability of $Q_{\la, k}$ is an immediate consequence; see \cite[Section $5$]{LOS1} for a detailed proof in the case $k = 1$. Indeed, one can show that for all $\Te < \Te_*$  and $\la = \Te^{1/k}$, there exists $\de_0>0$ so that for every $(\psi_0, \psi_1) \in \E_{\la,k}$ with
\ant{
 \|(\psi_0, \psi_1)- (Q_{\la,k}, 0)\|_{\HH_0}< \de_0
}
there exists a unique global solution $\vec \psi(t) \in \E_{\la,k}$ to~\eqref{wm}, which  \emph{scatters to $(Q_{\la,k}, 0)$} as $t \to \pm \infty$.


\end{rem}

\begin{rem}\label{rem:open} 
The arguments in Section~\ref{spectra} show that geodesic convexity of the image, which corresponds to $\la^k \le 1$, is a sufficient condition for the asymptotic stability of the harmonic maps $Q_{\la, k}$. However, Theorem~\ref{S2 no eval}$(ii)$ shows that this condition is not necessary. Indeed, in~\cite{LOS1} the authors showed that in the case $k=1$, the conclusions of Theorem~\ref{S2 no eval} hold for $0 \le \la< \sqrt{15/8}$ and numerical simulations suggest that the first failure of Theorem~\ref{S2 no eval} happens for some $\la_0 > 3.4$ in the $1$-equivariant case. 

\end{rem}

Next, we demonstrate an interesting change in the spectrum $\s(H_{V_{\la, k}})$, which causes a breakdown in the dispersive behavior of solutions to the linearized equation~\eqref{v eq}. Fixing the equivariance class $k \in \N$ and taking $\la$ to be a large number, 
 we establish the existence of a unique, simple eigenvalue $\mu_{\la, k}^2$ in the spectral gap $(0, 1/4)$. Moreover, we show that for each fixed class $k$, the eigenvalue $\mu_{\la, k}^2$ migrates to $0$ as $\la \to \infty$. Finally, we find a uniform in $k$ latitude $\al^*$ on the sphere so that for any equivariance class $k \in \N$ and harmonic map $Q_{\la, k}$ which contains $\al^*$ in its image, the corresponding linearized operator $H_{V_{\la, k}}$ has an eigenvalue in its spectral gap.   Note that the existence of $\mu_{\la, k}^2$ precludes a proof of the nonlinear asymptotic stability of $Q_{\la, k}$ via a purely linear dispersive  mechanism. 



\begin{thm}\label{e val} \begin{enumerate}[(i)] \item For each $k\geq1$ there exists $\Lambda_0=\Lambda_0(k)>0$ so that for all $\la > \La_0$, the Schr\"odinger operator $H_{V_{\la,k}}$ has a unique simple eigenvalue $\gapE$ in the spectral gap $(0, 1/4)$. That is, there exists a unique number $\gapE \in (0, 1/4)$ and a unique nonzero solution $\fy_\la \in L^2(\Hp^{2k+2})$ to
\EQ{\label{e vec}
H_{V_{\la,k}} \fy_{\la} =  \gapE \fy_{\la}.
}
The operator $H_{V_{\la, k}}$ has no threshold resonance at $1/4$ in the sense of Definition~\ref{res def} below. Moreover, the eigenvalue $\gapE$ migrates to $0$ as $\la \to \infty$, i.e.,
\EQ{ \label{mu to 0}
\gapE \to 0 \mas \la \to \infty.
}
\item Let  $\Theta:=\la^k$.  There exists $\Theta^*>0$ large enough, and independent of $k$,  so that for all $\Theta>\Theta^*$ the Schr\"odinger operator $H_{V_{\la,k}}$ (with $\la =\Te^{\frac{1}{k}}$) admits a gap eigenvalue in $(0,1/4)$. In particular, there exists an angle $\al^*< \pi$ large enough so that any $k$-equivariant harmonic map with  $$\lim_{r \to \infty} Q_{\la, k}(r) > \al^*$$  yields a linearized operator $H_{V_{\la, k}}$ with an eigenvalue in the spectral gap $(0, 1/4)$. 
\end{enumerate}
\end{thm}


\begin{rem}  \label{rem:res}
As in~\cite{LOS1} we note that, if we set  
\ant{
&\lambda_{\sup}(k):=\sup\{ \la \mid H_{V_{\ti \lambda,k}} \,  \textrm{has no e-vals and no threshold resonance} \,  \forall \, \ti \la < \la\} \\
&\Lambda_{\inf}(k):=\inf\{ \la \mid H_{V_{\ti \lambda,k}} \, \, \textrm{has a gap e-val}\, \,  \mu^2_{\ti{\la}} \in (0, 1/4), \, \,  \, \forall \,  \ti \la > \la\}
}
then both $H_{V_{\la_{\sup},k}}$ and $H_{V_{\Lambda_{\inf},k}}$ have a threshold resonance. Indeed, this is a consequence the proof of Theorem~\ref{e val} from which one can deduce that having, or not having, gap eigenvalues are open conditions in $\la$. See~\cite[Proposition~$3.6$]{LOS1} for details. We also remark that $\lmb_{\sup}$ and $\Lmb_{\inf}$ could very well coincide.
\end{rem}

\begin{rem}
We note that for $k = 1$, Theorem~\ref{e val} was proved in~\cite{LOS1}.  Here we address the remaining equivariance classes $k  \ge2$. The proof of the existence of gap eigenvalues for $k \ge 2$ given here is somewhat less involved than the existence proof given for the case $k=1$ in~\cite{LOS1}. This can be attributed to the fact that for $k\ge 2 $ the operators obtained by linearizing the Euclidean wave maps equation about $Q_{\euc, k}(r) = 2 \arctan(r^k)$ have threshold eigenvalues, while in the $k=1$ case the Euclidean linearized operator has a threshold resonance. 
Indeed, in the argument given in Section~\ref{s:eval}, we note that the final inequality~\eqref{fneg} \emph{does not} guarantee  a sign change for large $\la$ when $k=1$. In the $k=1$ case, the crucial sign change requires a delicate analysis of the last term in the first line of~\eqref{frho}, where a logarithmic divergence can be extracted, see~\cite[Section $3.3$]{LOS1}. 
\end{rem}

\begin{rem} \label{rem:FGR}
Finally, we remark that the existence of a gap eigenvalue for $H_{V_{\la, k}}$ means that one cannot prove Strichartz estimates for the the~\emph{linear} equation
\EQ{ \label{lin eq} 
(\p_t^2 + H_{V_{\la, k}})u = 0
} 
and hence asymptotic stability of $Q_{\la, k}$ cannot be established  directly via linear dispersive estimates together with a perturbative argument.  Indeed, if the eigenvalue and associated eignevector for $H_{V_{\la, k}}$ are given by $\mu_\la^2$ and $\fy_{\mu_{\la}}(r)$, the solution $u(t, r) = e^{\pm i t \mu_\la} \fy_{\mu_\la}(r)$ to~\eqref{lin eq}  does not disperse. 

On the other hand, there are several reasons  to believe (on a heuristic level) that \emph{every} $Q_{\la, k}$ is asymptotically stable in the energy topology.  In particular, we show in a forthcoming work that finite time blow-up results in the bubbling off of a $k$-equivariant Euclidean harmonic map, i.e., when say $k=1$, for any solution $\vec \psi(t)$ that blows up at time $t=1$, there exists a sequence of times  $t_n \to 1$, and a sequence of scales $\be_n = o(1-t_n)$ so that 
\ant{
\vec \psi(t_n) - (Q( \cdot/ \be_n), 0) \to 0 \mas n \to \infty
}
\emph{locally} in the energy topology. This means that $\vec \psi$ must have energy $\E> \E_{\euc}(Q_{\euc, k})$ in order to blow up. Since $Q_{\la, k}$ has energy strictly less than the Euclidean energy of $Q_{\euc, k}$, we know that any small perturbation of $Q_{\euc, k}$ is defined globally in time. Moreover, since $Q_{\la, k}$ minimizes the energy in the class $\E_{\la, k}$, one can follow the classical argument of Cazenave and  Lions~\cite{CL82} to prove the Lyapunov stability of $Q_{\la, k}$. 

The remaining question is the following: can one describe the asymptotic dynamics of small perturbations of $Q_{\la, k}$ when the linearized operator $H_{V_{\la, k}}$ has a gap eigenvalue? A possible approach to answering this question is given in the work of Soffer and Weinstein~\cite{SW99}, see also the subsequent developments~\cite{GS07, CM08, Miz08, BC11}. In the framework introduced in~\cite{SW99} the  decay of solutions to a nonlinear Klein Gordon equation with a potential -- chosen so that that the linearized operator has a simple gap eigenvalue -- is reduced to a verification of a so-called ``\emph{nonlinear Fermi Golden Rule}" condition, see~\cite[equation $(1.8)$ and Remark~$1$]{SW99}. In particular, if the nonlinear Fermi Golden Rule is satisfied, then constant multiples of the eigenfunction decay at a slow rate determined by how far the eigenvalue is from the edge of the continuous spectrum and via a mechanism known as \emph{radiative damping}. Here the decay comes from a purely nonlinear mechanism as  the nonlinearity of the equation forces the  energy corresponding to the eigenfunction to eventually leak into the continuous spectrum.  

The equations studied here are proposed as \emph{explicit} natural models in which to study the phenomena of dispersion by radiative damping near nontrivial stationary solutions. We remark that although the nonlinear Fermi Golden Rule is known to hold for generic potentials \cite{SW99, BC11}, there does not seem to be a convenient criterion for verifying this condition for specific potentials, such as those arising in this paper.

\end{rem}

\subsection{Equivariant Yang-Mills on $\R \times \Hp^4$}
Consider the Yang-Mills equations on $\R \times \Hp^4$ with gauge group $SU(2)$. 
The dynamic variable is a connection $A_\mu$, which is a one-form with values in $\mathfrak{su}(2)$, the Lie algebra of $SU(2)$.
Associated with $A$ we have the covariant derivative $D_\mu:=\nabla_\mu+[A_\mu,\cdot],$ where $[\cdot,\cdot]$ denotes the Lie bracket. In terms of the curvature two-form $F_{\mu\nu}:=\nabla_\mu A_\nu-\nabla_\nu A_\mu+[A_\mu,A_\nu]$ the Yang-Mills functional is given by\footnote{We adopt the usual convention of raising and lowering indices by the Minkowski metric, and summing repeated upper and lower indices.}
\ant{
\LL_{\mathrm{YM}}(F)=\int_{\Hp^4\times\R} F^{\mu\nu}F_{\mu\nu},
}
with associated Euler-Lagrange equations
\EQ{\label{eqF}
D^\mu F_{\mu\nu}=0.
}

We now formulate the notion of an equivariant Yang-Mills field, following the general framework described in~\cite[Section 2]{SSTZ}. For this purpose it is useful to use the Poincar\'e disk model of $\bbH^{4}$, which is related to the polar coordinates on $\bbH^{4}$ by the map
\begin{equation*}
	\Psi : \bbD \setminus \set{0} \ni x \mapsto (2 \, \arctanh\abs{x}, \frac{x}{\abs{x}}) \in (0, \infty) \times \bbS^{3},
\end{equation*}
where $\bbD := \set{x \in \bbR^{4} : \abs{x} < 1}$ and $\abs{x} := \sqrt{(x^{1})^{2} + \cdots + (x^{4})^{2}}$. The metric then takes the form
\begin{equation*}
	g_{\bbH^{4}} = \bb(\frac{2}{1-\abs{x}^{2}}\bb)^{2} g_{\bbD^{4}}, \quad 
	\hbox{ where } g_{\bbD^{4}} := (\ud x^{1})^{2} + \cdots + (\ud x^{4})^{2}.
\end{equation*}
The rotation group $SO(4)$ acts on $\bbH^{4}$ as isometries by the natural action on $\bbD^{4}$. To specify the action of the symmetry group on $\mathfrak{su}(2)$, however, we furthermore pass to the double cover\footnote{This covering can be conveniently realized in the language of quaternions; see \cite[Section 2.1]{SSTZ}} of $SO(4)$, namely the product group $SU(2) \times SU(2)$. The action of $SO(4)$ on $\bbH^{4}$ lifts to $SU(2) \times SU(2)$, and we use the projection $(p, q) \mapsto q$ to specify the action of $SU(2) \times SU(2)$ on $\mathfrak{su}(2)$. 

We say that a connection $A$ on $\bbR \times \bbH^{4}$, which is an $\mathfrak{su}(2)$-valued 1-form, is \emph{equivariant} under $SU(2) \times SU(2)$ if
\begin{equation} \label{eq:equiv4YM}
	A(t, x; v^{0}, v) = q^{-1} A(t, (p, q) \cdot x; v^{0}, (p, q) \cdot v) q
\end{equation}
for all $(t, x) \in \bbR \times \bbH^{4}$, $(v^{0}, v) \in T_{t,x} (\bbR \times \bbH^{4})$ and $(p, q) \in SU(2) \times SU(2)$, where $(p, q) \cdot x$ denotes the action of $SU(2) \times SU(2)$ on $\bbH^{4}$, and $(p, q) \cdot v$ is the induced push-forward $T_{x} \bbH^{4} \mapsto T_{(p, q) \cdot x} \bbH^{4}$.

By the same argument as in \cite[Section 2.2]{SSTZ}, the condition \eqref{eq:equiv4YM} leads to the following ansatz for an equivariant connection $A$ on $\bbR \times \bbH^{4}$:
\begin{equation*}
	A(t, \abs{x}, 0, 0, 0 ; v^{0}, v^{1}, \ldots, v^{4}) 
	= - \frac{\psi(t, \abs{x})}{\abs{x}} (v^{2} \sgm_{2} + v^{3} \sgm_{3} + v^{4} \sgm_{4}),
\end{equation*}
where $\psi$ is a radial function on $\bbR \times \bbH^{4}$ and $\set{\sgm_{2}, \sgm_{3}, \sgm_{4}}$ is an orthonormal basis for the Lie algebra $\mathfrak{su}(2)$ of complex-valued $2 \times 2$ anti-hermitian matrices with vanishing trace\footnote{A particular representation is
\begin{equation*}
	\sgm_{2} = \ttmatrix{i}{0}{0}{-i}, \quad
	\sgm_{3} = \ttmatrix{0}{1}{-1}{0}, \quad
	\sgm_{4} = \ttmatrix{0}{i}{i}{0}.
\end{equation*}
The bi-invariant inner product on $\mathfrak{su}(2)$ is defined by $(A, B) := \frac{1}{2} \tr(A B^{\ast})$.
We note that such $\sgm_{2}, \sgm_{3}, \sgm_{4}$ are related to the usual Pauli matrices $\sgm_{x}, \sgm_{y}, \sgm_{z}$ by $\sgm_{2} = \frac{1}{i} \sgm_{z}$, $\sgm_{3} = \frac{1}{i} \sgm_{y}$ and $\sgm_{4} = \frac{1}{i} \sgm_{x}$.}.
The ansatz is given only for points on the $x^{1}$-axis, but it can easily be extended to the whole $\bbR \times \bbH^{4}$ by the condition \eqref{eq:equiv4YM}. The reduced Lagrangian in the coordinates $(t, s) = (t, \abs{x})$ then reads
\begin{equation*}
	\calL_{\mathrm{YM}}(A) = 3 \omg_{3} \int_{\bbR} \int_{0}^{1} \bb( - \psi_{t}^{2} + \frac{(1-s^{2})^{2}}{4}\psi_{s}^{2} + \frac{(1-s^{2})^{2}}{4 s^{2}} ( \psi (2 - \psi ))^{2} \bb) \frac{2s}{1-s^{2}}\,\ud s \, \ud t,
\end{equation*}
where $\omg_{3}$ is the volume of the unit $3$-sphere.

Computing the Euler-Lagrange equation and making a change of coordinates $s \mapsto r = 2 \arctanh s$, the system~\eqref{eqF} reduces to the following semi-linear wave equation for $\psi$, 
\EQ{\label{YM reduced eq}
&\psi_{tt}-\psi_{rr}-\coth r \psi_r+4\frac{(\psi-\frac{\psi^2}{2})(1-\psi)}{\sinh^2r}=0,\\
&\vec{\psi}(0)=(\psi_0,\psi_1).
}
with conserved energy 
\ali{\label{YM energy}
\EE_{\YM}(\psi)=\frac{1}{2}\int_0^\infty\left(\psi_t^2+\psi_r^2+4\frac{(\psi-\frac{\psi^2}{2})^2}{\sinh^2r}\right)\sinh r dr.
}
The Euclidean equivariant Yang-Mills system on $\bbR^{1+4}$ can be similarly reduced to 
\EQ{\label{euc YM}
&\phi_{tt}-\phi_{rr}-\frac{1}{r} \phi_r+4\frac{(\phi-\frac{\phi^2}{2})(1-\phi)}{r^2}=0,\\
&\vec{\phi}(0)=(\phi_0,\phi_1).
}
with conserved energy 
\ali{\label{euc YM energy}
\EE_{\YM}^\euc(\phi)=\frac{1}{2}\int_0^\infty\left(\phi_t^2+\phi_r^2+4\frac{(\phi-\frac{\phi^2}{2})^2}{r^2}\right)r dr.
}

We remark that~\eqref{YM reduced eq} can be thought of as a 2-equivariant wave map equation into a surface of revolution with metric $dv^2+g_{\YM}^2(v)du^2$ where $g_{\YM}(v):=v-\frac{v^2}{2},$ at least for $v\in(0,2)$ where the metric is non-degenerate.  Similar to the  case of equivariant wave maps, the requirement that  the initial data $(\psi_0, \psi_1)$  have finite energy mandates that $\psi_0(0)  \in \{0, 2\}$. Here we restrict to the case $\psi_0(0) = 0$. However, the behavior at $r = \infty$ is more flexible. While the Euclidean equivariant Yang-Mills~\eqref{euc YM} and~\eqref{euc YM energy} require that  $\lim_{r\to \infty} \phi_0(r) \in \{0, 2\}$ in order for the energy to be finite, the hyperbolic space versions~\eqref{YM reduced eq} and~\eqref{YM energy} allow $ \psi_0(r)$ to take  any finite limit as $r \to \infty$.

 
 As in the case of wave maps, we can divide the energy space into disjoint classes, which depend on the endpoint of the initial data at $r = \infty$. Here we restrict to endpoints $ \psi_0( \infty) \in [0, 2)$, and define 
\ant{
\EE^\YM_{\lambda}:=\left\{(\psi_0,\psi_1)|\EE_\YM(\vec{\psi})<\infty,~\psi_0(0)=0,~\psi_0(\infty)=\frac{2\lambda^2}{1+\lambda^2}\right\}
}
for all  $\lambda\in[0,\infty)$.  The reason for considering only these endpoints is that in each class $\EE^{\YM}_{\la}$ there is a unique stationary solution 
\EQ{\label{YM solitons}
Q_{\YM, \la}(r):=\frac{2\lambda^2\tanh^2(r/2)}{1+\lambda^2\tanh^2(r/2)}
} 
which minimizes the energy in this class. We note that $Q_{\YM, \la}$ solves, 
\ali{\label{YM elliptic eq}
&Q_{rr}+\coth r \, Q_r=4\frac{g_{\YM}(Q)g_{\YM}^\prime(Q)}{\sinh^2r}, \quad Q(0)=0,
}
with 
\EQ{
 \lim_{r\to\infty}Q(r)=\frac{2\lambda^2}{1+\lambda^2}
 }
 In Section~\ref{sec:pre}, we show that $  Q_{\YM, \la}$ are the only possible, finite energy  stationary solutions of \eqref{YM elliptic eq} and have energies
\ant{
\EE_\YM(Q_{\YM, \la},0)=\frac{4\lambda^4(3+\lambda^2)}{3(1+\lambda^2)^3} 
}
which is minimal in  $\EE_\la^\YM.$ 

Recall that the unique (up to scaling) nontrivial, finite energy stationary solution  to the Euclidean problem~\eqref{euc YM} is given by 
\ant{
Q_{ \YM, \euc}(r):=\frac{2r^2}{1+r^2}.
}
We have $\EE_\YM^\euc(Q_{\YM, \euc})=\frac{4}{3}$, which is minimal  amongst solutions starting at $\phi(t,0)=0$ and ending at $\phi(t,\infty)=2$.  We remark that   
\ant{
\EE_\YM(Q_{\YM, \la},0)  \to \EE_\YM^\euc(Q_{\YM, \euc})=\frac{4}{3}  \mas  \la\to\infty 
}
and $\EE_\YM(Q_{\YM, \la},0)  \to 0 $ as $\la \to 0$.

\vspace{\baselineskip}


The setup for investigating the asymptotic stability of $Q_{\YM, \la}$ under~\eqref{YM reduced eq} is similar to that of the $k$-equivariant harmonic map $\Qu$ presented in the previous subsection with $k=2$.  For a solution $\vec{\psi}\in\EE_\la^\YM$ to  \eqref{YM reduced eq} we define $\vec u(t, r)$ by 
 \ant{
\sinh^2r u(t,r):=\psi(t,r)-Q_{\YM, \la}(r).
}
Then $\vec{u}(t)$ solves
\ali{\label{u eq YM}
&u_{tt}-u_{rr}-5\coth ru_r-6u+W_\lambda(r)u=\NN_\YM(r,u)\\
&\vec{u}(0)=(u_0,u_1)
}
where the attractive potential $W_\la$ is given by 
\ant{
W_\lambda=\frac{6Q_{\YM, \la}(r)\left(Q_{\YM, \la}(r)-2\right)}{\sinh^2r}\leq0,\\
}
and the nonlinearity $\NN_{\YM}$ is 
\ant{
\NN_\YM(r,u)=-\frac{4}{\sinh^4r}\left(\frac{1}{2}\sinh^6ru^3+\frac{3}{2}(Q_{\YM, \la}-1)\sinh^4ru^4\right).
}
The underlying \emph{linear} equation is then, 
\EQ{ \label{v eq2}
v_{tt} - \Delta_{\Hp^{6}} v -6  v + W_{\la} v = 0
}
for radially symmetric functions $v(t):  \Hp^{6} \to \R$, and we define self-adjoint Schr\"odinger operators
\EQ{\label{H Wla}
&H_{0,2}:=-\p_{rr}- 5 \coth r \, \p_r -6,\\
&H_{W_{\la}} := -\p_{rr}- 5\coth r \, \p_r -6 +W_{\la}.
}
where we have used the notation $H_{0, 2}$ for the free operator in analogy with~\eqref{H Vla} -- here argain note that $\s(H_{0, 2}) = [1/4,  \infty)$ where $1/4$ is neither an eigenvalue nor a resonance. 

As for $k$-equivariant wave maps, the goal here will be to understand the spectrum $ \s(H_{W_{\la}})$. First, we prove that for maps $Q_{\YM, \la}$, whose image is contained  in a geodesically convex neighborhood of $Q_{\YM, \la}(0)$, the spectrum of the linearized operator $H_{W_\la}$ is purely absolutely continuous and is given by $[1/4, \infty)$. 
\begin{thm}\label{YM no eval}
For each  $0 \le \lambda\leq1$,   the spectrum of $H_{W_{\la}}$ is purely absolutely continuous and is given by  $$\sigma(H_{W_{\la}})=[1/4,\infty).$$  In particular, $H_{W_{\la}}$ has no negative spectrum, no eigenvalues in the spectral gap $[0,1/4),$ and the threshold $1/4$ is neither a resonance nor an eigenvalue.
\end{thm}

\begin{rem}\label{rem:ymstab} 
As in the case of wave maps, one can use the spectral information in Theorem~\ref{YM no eval} to deduce Strichartz estimates for~\eqref{v eq2}; see Remark~\ref{rem:stab}. As a consequence one can prove  the following stability result via the usual arguments.
 For every $\la \in [0, 1]$ the harmonic map $Q_{\YM, \la}$ is asymptotically stable in $\EE_{\la}^\YM$. In particular, for each $\la \in [0, 1]$ there exists a $\de_0>0$ so that for every $(\psi_0, \psi_1) \in \EE_{\la}^\YM$ with
\ant{
 \| (\psi_0, \psi_1) - (Q_{\YM, \la}, 0)\|_{\HH_0} < \de_0
}
there exists a unique, global solution $\vec \psi(t) \in \EE_{\la}^\YM$ to~\eqref{YM reduced eq}. Moreover, $\vec \psi(t)$  scatters to $(Q_{\YM, \la}, 0)$  as $t \to \pm \infty$.
\end{rem}

\begin{rem}
As in Remark~\ref{rem:open} we note that  Theorem \ref{YM no eval} and Remark~\ref{rem:ymstab} can be extended to include $\la$ with  $\la<1+\epsilon$ for a sufficiently small $\epsilon>0.$
\end{rem}

For $\la$ large the situation changes dramatically, as we have the following analog of Theorem \ref{e val}, for the Yang-Mills problem. 

\begin{thm}\label{YM e val} There exists $\Lambda_1>0$ so that for all $\la > \La_1$, the Schr\"odinger operator $H_{W_{\la}}$ has a unique simple eigenvalue $\mu_{\lmb}^{2}$ in the spectral gap $(0, 1/4)$. That is, there exists a unique number $\mu_{\lmb}^{2} \in (0, 1/4)$ and a unique nonzero solution $\fy_\la \in L^2(\Hp^6)$ to
\EQ{
H_{W_{\la}} \fy_{\la} =  \mu_{\lmb}^{2} \fy_{\la}.
}
The operator $H_{W_{\la}}$ has no threshold resonance at $1/4$ in the sense of Definition~\ref{res def} below. Moreover, the eigenvalue $\mu_{\lmb}^{2}$ migrates to $0$ as $\lmb \to \infty$, i.e.,
\EQ{ \label{mu to 0 ym}
\mu_{\lmb}^{2} \to 0 \mas \la \to \infty.
}
\end{thm}

\begin{rem}
As in Remark~\ref{rem:res}, a consequence of the proofs of Theorem~\ref{YM no eval} and~\ref{YM e val}, is that if we define 
\EQ{
&\lambda_{\sup}:=\sup\{ \la \mid H_{W_{\ti \lambda}} \, \, \textrm{has no e-vals and no threshold resonance} \, \, \forall \, \ti \la < \la\}\\
&\Lambda_{\inf}:=\inf\{ \la \mid H_{W_{\ti \lambda}} \, \, \textrm{has a gap e-val}\, \,  \mu^2_{\ti{\la}} \in (0, 1/4) \, \, \forall \,  \ti \la > \la\}
}
then both $H_{W_{\la_{\sup}}}$ and $H_{W_{\Lambda_{\inf}}}$ have threshold resonances; again we refer the reader to~\cite[Proposition~$3.6$]{LOS1} for more details. 
\end{rem}
\subsection{Brief outline of the paper}
In Section~\ref{sec:pre} we establish  various facts about the harmonic maps $Q_{\la,k},$ and $Q_{\YM, \la}$ defined above.  We also justify  the passage to equations on $\R \times \Hp^{2k+2}$ outlined in the introduction.

In Section~\ref{spectra} we begin the study of the spectra of the linearized operators $H_{V_{\la,k}}$ and $H_{W_{\la}}$ and prove Theorem~\ref{S2 no eval} and Theorem~\ref{YM no eval}. In Section~\ref{s:eval} we prove Theorem~\ref{e val} and Theorem~\ref{YM e val}. 

\section{Preliminaries}\label{sec:pre} In this section we establish the existence and uniqueness of the harmonic maps $Q_{\la,k}$ and $Q_{\YM, \la}$ described in the introduction. We give simple geometric descriptions of these maps and prove several properties that we will need in the ensuing arguments. We also prove some additional preliminary facts including an equivalence between the $2d$ and $(2k+2)$ dimensional Cauchy problems described in the introduction.

\subsection{Basic properties of the stationary solutions $Q_{\la,k}$ and $Q_{\YM, \la}$}\label{S2 hm} Here we prove various facts about the harmonic maps $Q_{\la,k}$ and $Q_{\YM, \la}$. For convenience we collect these facts into two propositions.
\begin{prop}\label{hm prop}
For every $0 \le \al< \pi$ there exists a unique, finite energy stationary solution to~\eqref{wm}, i.e.,  a harmonic map, $(Q_{\la,k}, 0) \in \E_{\la,k}$ which solves~\eqref{hm}, where
\EQ{\label{Q def}
&Q_{\la,k}(r) = 2\arctan(\la^k \tanh^k(r/2))\\
&\la \in [0, \infty), \, \, \al = \al(\la) = 2 \arctan(\la^k) = \lim_{r \to \infty} Q_{\la,k}(r).
}
Moreover, $(Q_{\la,k}, 0) \in \E_{\la,k}$ has energy
\EQ{\label{hm en}
\E( Q_{\la,k}, 0) = 2 k\frac{\la^2}{1+\la^2},
}
which is minimal in $\E_{\la,k}$. Finally, the $Q_{\la,k}$ with $\la\in[0, \infty)$ are the only finite energy stationary solutions to~\eqref{wm}.
\end{prop}

\begin{proof}
We are seeking to classify all stationary finite energy solutions to~\eqref{wm}. Recall from the introduction that  any finite energy harmonic map $Q$ must have $Q(0) = 0$ and $Q(\infty) = \al \in [0 , \infty)$. Thus we would like to find all solutions $Q$ to
\EQ{\label{hm bd}
&Q_{rr} + \coth r \, Q_r =k^2 \frac{\sin2 Q}{2\sinh^2 r},\\
&Q(0) = 0, \, \, \lim_{r \to \infty}Q(r) = \al \in [0, \infty).
}
One can check directly that $Q_{\la,k}$, as defined in~\eqref{Q def}, satisfies~\eqref{hm bd} with~$\al = 2 \arctan(\la^k)$. One can also directly compute the energy to verify~\eqref{hm en}.

To prove the remaining statements in Proposition~\ref{hm prop} we begin by giving a simple geometric interpretation of $Q_{\la,k}$. Recall that the stereographic projection of $\Hp^2$ onto the Poincar\'e disc, $\D$, viewed as a subset of $\R^2$,  is given by 
\begin{align*}
( \sinh r \, \cos \omega, \sinh r\, \sin \om, \cosh r ) \mapsto ( \tanh(r/2) \cos \omega,  \tanh(r/2) \, \sin \om) .
\end{align*}
Next, rescale the disc by $\la \in [0, \infty)$ via
\begin{align*}
( \tanh(r/2) \cos \om,  \tanh(r/2) \, \sin \om) \mapsto ( \la\tanh(r/2) \cos \om,  \la \tanh(r/2) \, \sin \om).
\end{align*}
Since we are interested in $k$-equivariant harmonic maps, we note that in polar coordinates the degree $k$ map $z\mapsto z^k$ from $\C$ to $\C$ is given by 
\ant{
(\rho \cos \om, \rho \sin \om) \mapsto (\rho^k\cos (k\omega),\rho^k\sin(k\omega)).
}
Finally, recall that the inverse of stereographic projection, $\R^2 \to \Sp^2-\{\textrm{south pole}\}$ is the map
\begin{align*}
(\rho \cos \omega, \rho \sin \omega) \mapsto ( \sin(2\arctan \rho) \, \cos \omega, \sin(2\arctan\rho) \,  \sin \omega, \cos(2 \arctan\rho)).
\end{align*}
Then since solutions to~\eqref{wm} or~\eqref{hm} are expressed in terms of the polar angle on $\Sp^2$,  we see that  $Q_{\la,k}$ is simply the composition of the above four maps. 

This geometric interpretation motivates the following change of variables in~\eqref{hm bd}. If we define
\ant{
s:= \log( \tanh(r/2)), \, \, \fy(s):= Q(r),
}
then~\eqref{hm bd} reduces to the following equation for $\fy$,
\EQ{\label{ode}
 &\fy_{ss} =  \frac{k^2}{2} \sin 2 \fy,\\
 & \fy(- \infty) = 0, \, \, \fy(0) = \al.
 }
Multiplying the first line in~\eqref{ode} by $\fy'$ and integrating from $s_1$ to $s_2$ yields the energy identity
\EQ{\label{ode en id}
\fy_s^2(s_2) - \fy_s^2(s_1)= k^2\left(\sin^2(\fy(s_2))- \sin^2(\fy(s_1))\right).
}
A standard analysis of the phase portrait  in $(\fy, \fy')$ coordinates together with~\eqref{ode en id} shows that any nontrivial solution must satisfy $0< \fy(0)< \pi$. Thus there are no nontrivial solutions for endpoints $\al \ge \pi$, i.e,  $Q_{\la, k}$ are the only solutions to~\eqref{hm bd}. 

It remains to show that $(Q_{\la,k}, 0)$ minimizes the energy in $\E_{\la,k}$. This is a direct consequence of the following ``Bogomol'nyi factorization": Let $\vec \psi(t) = ( \psi(t), \psi_t(t)) \in \E_{\la,k}$. Then we have
 \ant{
 &\E( \vec \psi) = \frac{1}{2} \int_0^{\infty} \psi_t^2 \, \sinh r \, dr + \frac{1}{2}\int_{0}^{\infty} \left( \psi_r - k\frac{\sin \psi}{ \sinh r} \right)^2 \, \sinh r\, dr + k\int_{0}^{\infty}  \sin \psi \psi_r \, dr\\
 & = \frac{1}{2} \int_0^{\infty} \psi_t^2 \, \sinh r \, dr + \frac{1}{2}\int_{0}^{\infty} \left( \psi_r - k\frac{\sin \psi}{ \sinh r} \right)^2 \, \sinh r\, dr\\&\qquad\qquad\qquad\qquad\qquad\qquad\qquad\qquad\qquad\qquad\qquad + k(\cos \psi(t, 0) - \cos \psi(t,  \infty))\\
 & = \frac{1}{2} \int_0^{\infty} \psi_t^2 \, \sinh r \, dr + \frac{1}{2}\int_{0}^{\infty} \left( \psi_r - k\frac{\sin \psi}{ \sinh r} \right)^2 \, \sinh r\, dr\\&\qquad\qquad\qquad\qquad\qquad\qquad\qquad\qquad\qquad\qquad\qquad + k(1- \cos(2 \arctan( \la))).
 }
 For the solution $\vec \psi(t) = (Q_{\la,k}, 0)$ the first two integrals -- which we note are always  non-negative -- vanish identically, which proves that $(Q_{\la,k}, 0)$ uniquely minimizes the energy in $\E_{\la,k}$. Finally, a simple  calculation yields
 \ant{
 \E(Q_{\la,k}, 0) =k(1- \cos(2 \arctan( \la))) = 2k \frac{\la^2}{1+ \la^2}
 }
 and this completes the proof.
\end{proof}

The analogous result for 
$Q_{\YM, \la}$ can be proved in a nearly identical fashion. We omit the details.  


\begin{prop}\label{hm prop ym}
For every $0 \le \al< 2$ there exists a unique, finite energy stationary solution to~\eqref{YM reduced eq}, i.e.,  a map, $(Q_{\YM, \la}, 0) \in \EE_{\la}^\YM$ which solves~\eqref{YM elliptic eq}, where
\EQ{\label{Y def}
&Q_{\YM, \la}(r) = \frac{2\la^2\tanh^2(r/2)}{1+\la^2\tanh^2(r/2)}\\
&\la \in [0, \infty), \, \, \al = \al(\la) = \frac{2\lambda^2}{1+\la^2} = \lim_{r \to \infty} Q_{\YM, \la}(r).
}
Moreover, $(Q_{\YM, \la}, 0) \in \EE_{\la}^\YM$ has energy
\EQ{\label{hm en ym}
\EE_\YM( Q_{\YM, \la}, 0) = \frac{4\lambda^4(3+\lambda^2)}{3(1+\lambda^2)^3},
}
which is minimal in $\E_{\la}^\YM$. Finally, the $Q_{\YM, \la}$ with $\la\in[0, \infty)$ are the only finite energy stationary solutions to~\eqref{YM reduced eq}.
\end{prop}

\subsection{Reduction to equations on $\R \times \Hp^{2k+2}$}\label{4d}
Next, we provide more details related to the reductions to the Cauchy problems~\eqref{u eq} 
and~\eqref{u eq YM}  outlined in the introduction.

First, we prove an $L^{\infty}$ bound on solutions to~\eqref{wm} 
and \eqref{YM reduced eq} in terms of their energy. As the proof is the same in both cases we shorten the exposition by considering solutions to~\eqref{eq wm}, namely
\EQ{ \label{wm g}
&\psi_{tt} - \psi_{rr} - \coth r \, \psi_r + k^2\frac{ g(\psi)g'(\psi)}{\sinh^2 r} = 0,\\
&E( \vec \psi) :=  \frac{1}{2}\int_0^\infty  \left( \psi_t^2 + \psi^2_r + k^2\frac{g^2( \psi)}{\sinh^2 r} \right) \sinh r \, dr.
}
where in the cases under consideration we have $g_{\Sp^2}(\psi) = \sin \psi$, $E= \E$ for maps into $\Sp^2$, and $g_{\YM}(\psi)= \,\psi - \frac{\psi^2}{2}$, $E=\EE_\YM$, $k=2$ for the Yang-Mills equation.
\begin{lem} Let $\vec \psi(t)$ be a finite energy solution to~\eqref{wm g} defined on the interval $t \in I$ with $\psi(t, 0) = 0$ for every $t \in I$. Then there exists a function $C$ with $C( \rho) \to 0$ as $\rho \to 0$ so that
\EQ{\label{l inf bound}
\sup_{t \in I} \| \psi(t) \|_{L^{\infty}} \le C( E( \vec \psi)).
}
\end{lem}
\begin{proof} Following, e.g.,~\cite[Chapter $8$]{SSbook} we define the function
\ant{
G(\psi)  = \int_{0}^{\psi} \abs{g(\rho)} \, d \rho,
}
and we note that $G(0) = 0$, $G$ is increasing, and $G(\psi) \to \infty$ as $\psi \to \infty$. 
For any fixed $t \in I$ we have
\EQ{\label{G bound}
\abs{G( \psi(t, r)) }&= \abs{G( \psi(t,r)) - G( \psi(t, 0))} = \abs{\int_{\psi(t, 0)}^{\psi(t, r)} \abs{g (\rho)} \, d \rho}\\
& = \int_0^{r} \abs{ g( \psi(t, r))} \abs{ \psi_r(t, r)} \, dr  \le E( \vec \psi)\\
}
Then~\eqref{l inf bound} follows from \eqref{G bound} and the fact that $G$ is increasing. 
\end{proof}

Next, we establish  an equivalence of the Cauchy problems~\eqref{wm} with~\eqref{u eq}, and \eqref{YM reduced eq} with \eqref{u eq YM},  by providing  an isomorphism between the spaces $\HH_0$ and $H^1 \times L^2( \Hp^{2k+2})$, where $\HH_0$ is defined as in~\eqref{H0 def} and where for radially symmetric $u, v : \Hp^{2k+2} \to \R$ we set
\ant{
\|(u, v)\|_{ H^1 \times L^2( \Hp^{2k+2})}^2:= \int_0^{\infty} \left(u_r^2(r)+ v^2(r) \right) \sinh^{2k+1} r \, dr.
}
We use the notation $H^1$ for the above norm as opposed to $\dot{H}^1$ due to the embedding~
\EQ{\label{H1L2}
\dot{H}^1( \Hp^d) \hookrightarrow L^2( \Hp^d)
} for $d \ge 2$, see for example~\cite{Bray}. 
We prove the following simple lemma.
\begin{lem}\label{2d to 4d}Let $(\psi, \phi) \in \HH_0(\Hp^2)$ with $\psi(0)=0$, $\psi( \infty) = 0$. Then if we define $(u, v)$ by
\ant{
 (\psi(r), \phi(r)) = ( \sinh ^kr \,  u(r), \sinh ^kr \,  v(r))
 }
  we have
  \EQ{\label{H=H}
  \| (\psi, \phi) \|_{\HH_0}^2 \le  \|(u, v) \|_{H^1 \times L^2(\Hp^{2k+2})}^2 \le C(k) \|(\psi, \phi) \|_{\HH_0}^2.
}
\end{lem}


\begin{proof} Since $\|\phi\|_{L^2(\Hp^2)}^2= \|v \|_{L^2( \Hp^{2k+2})}^2$ it suffices to just consider $u$ and $\psi = \sinh^k r u$. Integration by parts yields the following identity:
\EQ{\label{2d 4d}
\int_0^{\infty} \left(\psi_r^2 + k^2\frac{\psi^2}{\sinh^2 r} \right) \, \sinh r \, dr =  \int_0^{\infty}( u_r^2-k(k+1)u^2 )\sinh^{2k+1} r\, dr.
}
The lemma now follows from the above and \eqref{H1L2}
\end{proof}


\section{The operators $H_{V_{\la,k}}$ and $H_{W_\la}$:  Analysis of the Spectrum}\label{spectra}
In this section we give a detailed analysis of the spectrum of the~Schr\"odinger operators $H_{V_{\la,k}}$ and $H_{W_\la}$ defined in~\eqref{H Vla} and \eqref{H Wla}, which are  self-adjoint on the domain $\calD_k: = H^2( \Hp^{2k+2})$, restricted to radial functions. 

In Section~\ref{subsec:smallLmb}, we prove Theorem~\ref{S2 no eval}$(i)$ and Theorem~\ref{YM no eval}. Specifically, we prove that for $\lambda\leq1$ the spectra of $H_{V_{\la,k}}$ and $H_{W_\la}$ coincide with that of the unperturbed operator $H_{0,k} := - \Delta_{\Hp^{2k+2}} - k(k+1)$ -- note that here for the Yang-Mills problem~\eqref{H Wla} we have $k =2$. 
In Section~\ref{s:k}, we establish Theorem~\ref{S2 no eval}$(ii)$ concerning the large $k$ case.
After showing in Section~\ref{s:noneg} that any eigenvalues for $H_{V_{\la,k}}$ or $H_{W_\la}$ must lie in the spectral gap $(0, 1/4)$, we devote the remainder of the paper, i.e., Section~\ref{s:eval}, to proving Theorem~\ref{e val} and Theorem~\ref{YM e val}. That is, we show that for $\la$ large enough, there is a unique simple eigenvalue $\mu_{\la}^2$ in the spectral gap $(0, 1/4)$ with no threshold resonance at $1/4$ and $\mu_\la^2 \to 0$ as $\la \to \infty$. 

First we pass to the half-line by conjugating  by $\sinh^{k+\frac{1}{2}} r$.  Indeed,  the map
\EQ{ \label{4 to 1}
L^2(\Hp^{2k+2}) \ni \fy\mapsto \sinh^{k+\frac{1}{2}} r\, \fy =: \phi \in L^2(0, \infty)
}
is an isomorphism of $L^2(\Hp^{2k+2})$, restricted to radial functions, with $L^2([0, \infty))$. If we define $\LL_0, \LL_{V}$ by
\EQ{ \label{LLdef}
&\LL_0 := - \p_{rr} + \frac{1}{4} + \frac{(k^2-1/4)}{\sinh^2 r},\\
&\LL_{V} := - \p_{rr} + \frac{1}{4} + \frac{(k^2-1/4)}{\sinh^2 r} + V(r),
}
we have
\EQ{ \label{conj}
&(H_{0,k} \fy)(r) = \sinh^{-k-\frac{1}{2}}r (\LL_0  \phi)(r),\\
&(H_{V} \fy)(r) = \sinh^{-k-\frac{1}{2}} r(\LL_{V} \phi)(r).
}
Hence it suffices to work with $\LL_0,$ $\LL_{V_{\la,k}},$ and $\LL_{W_\la}$ on the half-line. 
 We observe a few preliminary facts concerning solutions to
\EQ{\label{LL ev}
\LL_{\VV} \phi =  \mu^2 \phi, \, \, \mfor \mu^2 \in \R, \, \,  \mu \in \C,\qquad \VV=V_{\lambda,k}\mathrm{~or~}W_\lambda.
}

\begin{lem} \label{LL ev 0} Let $\mu \in \C, \, \, \mu^2 \in \R$ and suppose  $\phi_{\mu}$ is a solution to~\eqref{LL ev} with $\phi_{\mu} \in L^2([0, c))$ for some $c>0$. Then there exists a number $a \in \R$ such that
\begin{align} 
\phi_{\mu}(r) =& a\,  r^{k+\frac{1}{2}} + o(r^{k+\frac{1}{2}})  \mas r \to 0, \\
\phi_{\mu}'(r) =& a (k+\frac{1}{2}) r^{k-\frac{1}{2}} + o(r^{k-\frac{1}{2}})  \mas r \to 0.
\end{align}
\end{lem}
\begin{proof} This follows from the fact that the operator $\LL_0- 1/4$ can be approximated near $r=0$ by the singular operator
\ant{
L_0 := - \p_{rr} + \frac{(k^2-1/4)}{r^2}.
}
Note that $L_0$ is ``limit point"  at $r=0$ and a fundamental system for $L_0f = 0$ is given by $\{r^{k+\frac{1}{2}}, r^{-k+\frac{1}{2}}\}$. Therefore a solution $\phi_{\mu}$ as in~\eqref{LL ev} can be written in terms of this fundamental system by way of the variation of parameters formula, which converges for small $r$. The $L^2([0, c))$ requirement  ensures that the coefficient in front of $r^{-k+\frac{1}{2}}$ must be $0$ and the leading order behavior is given by $r^{k+\frac{1}{2}}$.
\end{proof}
\begin{lem} \label{LL ev 1}Suppose $\phi_{0}$ is a solution to~\eqref{LL ev} with $\mu^2 = \frac{1}{4}$. Then there are constants $a, b \in \R$ so that
\begin{align} 
\phi_0(r) =& a + b \, r + O(re^{-2r}) \mas r \to \infty, \label{phiinf} \\
\phi_0'(r) =& b + O(re^{-2r}) \mas r \to \infty.
\end{align}
\end{lem}
\begin{proof} This follows from the fact that we can find constants $C_{\la,k}, C_k>0$ so that for $r$ large we have $\VV(r) \le C_{\la,k} e^{-2r}$ and $\frac{(k^2-1/4)}{\sinh^2 r} \le C_k e^{-2r}$. Therefore the operator
\ant{
L_{\infty} := -\p_{rr}
}
is a good approximation of  $\LL_{\VV} - 1/4$ near $r= \infty$. A fundamental system for $L_{\infty} f = 0$ is  given by $\{1, r\}$, and the conclusions of Lemma~\ref{LL ev 1} follow from the variation of parameters formula.  
\end{proof}

Given the conclusions of Lemma~\ref{LL ev 0} and Lemma~\ref{LL ev 1} we can now precisely  define the term \emph{threshold resonance}.
\begin{defn}\label{res def} We say that $\phi_0$ is a \emph{threshold resonance} for $\LL_{\VV},$ $\VV=V_{\la,k}\mathrm{~or~}W_\la,$ if $\phi_0$ is \emph{not} in $L^2(0,\infty)$, but is a bounded solution to
\ant{
\LL_{\VV} \phi_0 = \frac{1}{4} \phi_0.
}
In particular we can find non-zero numbers $a, b \in \R$ so that
\ant{
&\phi_0(r) = a r^{k+\frac{1}{2}} + o(r^{k+\frac{1}{2}}) \mas r \to 0,\\
&\phi_0(r) = b + O(r e^{-2r}) \mas r \to \infty.
}
\end{defn}

With these preliminary facts in hand, we are ready to begin the proofs of Theorem~\ref{S2 no eval} and Theorem~\ref{YM no eval}.
\subsection{Spectrum of $H_{V_{\lmb,k}}$ or $H_{W_\la}$ for $\lmb \leq1$} \label{subsec:smallLmb} In this subsection we prove Theorem~\ref{S2 no eval}$(i)$ and Theorem~\ref{YM no eval}, modulo the statement that there are no embedded eigenvalues in the continuous spectrum, which will be established in Proposition~\ref{no neg spec}$(ii)$.
We note that the spectrum for the self-adjoint operator $\LL_0$ defined above is purely absolutely continuous and is given by $\s(\LL_0) = [1/4, \infty)$, and in particular there is no negative spectrum, no eigenvalue in the gap $[0, 1/4)$, and the threshold $1/4$ is neither an eigenvalue nor a resonance. 
In light of this discussion, to prove Theorems \ref{S2 no eval} and \ref{YM no eval} it suffices to show that in the case $0 \leq \la \leq 1$ the same can be said of the spectra $ \s(\LL_{V_{\la,k}})$ and $\sigma(\LL_{W_\la})$.  In particular, there is no negative spectrum, there are no eigenvalues in the gap $[0, 1/4)$, and the threshold $\frac{1}{4}$ is neither an eigenvalue nor a resonance.

We record a pair of alternative expressions for $\Lline$ and $\LL_{W_\la}.$ 
Using the notation $g_{\Sp^2}(x):=\sin x$ and $g_\YM(x)=x-\frac{x^2}{2}$ we have
\EQ{\label{Lline alt expression}
&\Lline := - \p_{rr} + \frac{1}{4} - \frac{1}{4\sinh^2 r} + k^2\frac{(g_{\Sp^2}^\prime(\Qu))^2+g_{\Sp}(\Qu)g_{\Sp}^{\prime\prime}(\Qu)}{\sinh^2r}, \\
&\LL_{W_\la}:=- \p_{rr} + \frac{1}{4} - \frac{1}{4\sinh^2 r} + 4\frac{(g_{\YM^2}^\prime(Q_{\YM, \la}))^2+g_{\YM}(Q_{\YM, \la})g_{\YM}^{\prime\prime}(Q_{\YM, \la})}{\sinh^2r}.
}

 To unify notation we use $g$ to denote $g_{\Sp^2}$ or $g_{\YM}$ and $Q$ to denote $\Qu$ or $Q_{\YM, \la}$ depending on the context. Note that in this notation $Q$ satisfies $$Q^\prime=k\frac{g(Q)}{\sinh r}.$$
\begin{proof}[Proof of Theorem~\ref{S2 no eval}$(i)$ and Theorem~\ref{YM no eval}]
Let $\mu \in \C$ with $\mu^2 \le \frac{1}{4}$. Suppose that $ \VV=V_{\la,k}$ or $W_\la,$ and $\phi_\mu$ is a solution to
\EQ{\label{phi m}
\LL_{\VV} \phi_\mu = \mu^2 \phi_{\mu}.
}
If  $\mu^2 \le 1/4$ is an eigenvalue, we  assume  that it is the smallest eigenvalue, and by a variational principle, we can further assume that the corresponding eigenfunction $\phi_{\mu} \in L^2$ is unique, (i.e., $\mu^2$ is simple) and strictly positive. If $\mu^2 = \frac{1}{4}$ and is not an eigenvalue, we  assume that $\phi_{\mu}$ is a
non-negative threshold resonance. In both cases, we know by Lemma~\ref{LL ev 0} that $\phi_{\mu}(r) = O(r^{k+\frac{1}{2}})$ as $r \to 0$. If $\phi_{\mu}$ is an eigenvalue, then $\phi_{\mu}(r),~\phi^\prime_{\mu}(r) \to 0$ as $r \to \infty$. If $\phi_{\mu}(r)$ is a threshold resonance, we know by Definition~\ref{res def} that $\phi_\mu(r) \to b >0$, and $\phi^\prime_\mu(r) \to 0$ as $r \to \infty$. Now, define a function $V(r)$ by 
\ant{
\LL_\VV-\frac{1}{4}&=-\partial_r^2+\left(k^2\frac{\left(g^\prime(Q)\right)^2+g(Q)g^{\prime\prime}(Q)}{\sinh^2r}-\frac{1}{4\sinh^2r}\right) =:-\partial_r^2+V.
}
In anticipation of applying a version of the Sturm comparison principle, we seek a positive function $f$ and a potential  $U$  satisfying
\ant{
(- \p_{r}^2 + U)f = 0, \quad  \textrm{with}  \, \, \, 
V-U\geq0.
}
If we define 
\EQ{
f(r):= g(Q) = k \sinh r \,  Q'(r)
}
then 
\ant{
f^{\prime\prime}=Uf,
}
with 
\ant{
U(r):=k^2\frac{\left(g^\prime(Q)\right)^2+g(Q)g^{\prime\prime}(Q)}{\sinh^2r}-\frac{kg^\prime(Q)\cosh r}{\sinh^2r}.
}
We claim that 
\EQ{\label{v-u}
V-U=\frac{1}{\sinh^2r}\left(k\cosh rg^\prime(Q)-\frac{1}{4}\right) \ge \frac{3}{4 \sinh^2r}
}
To prove~\eqref{v-u} we treat the cases $g=g_{\Sp^2},~Q=\Qu$ and $g=g_\YM,~Q=Q_{\YM, \la}$ separately.  Note that for $\la \le 1$ we have
\ant{
\la^{2k} \tanh^{2k}(r/2) \le \tanh^2(r/2) \mfor \forall k \ge 1, \, \, \, r \ge 0
}
Therefore, 
\ant{
\cosh rg^\prime(\Qu)  &=  \cosh r\left(\frac{1-\la^{2k}\tanh^{2k}(r/2)}{1+ \la^{2k}\tanh^{2k}(r/2)}\right) \ge  \cosh r\left(\frac{1-\tanh^{2}(r/2)}{1+ \tanh^2(r/2)}\right)  = 1
}
which proves~\eqref{v-u} for $Q = \Qu$. 
For the Yang-Mills problem using the formula $g_\YM^\prime(v)=1-v$ we have
\ant{
\cosh r \, g_\YM^\prime(Q_{\YM, \la})=\frac{\cosh r\left(1-\la^2\tanh^2(r/2)\right)}{1+\la^2\tanh^2(r/2)}   \ge   \cosh r\left(\frac{1-\tanh^{2}(r/2)}{1+ \tanh^2(r/2)}\right) = 1
}
which proves~\eqref{v-u} in the case $Q = Q_\YM$.  Now, since $\mu^2- 1/4 \le 0$ and since $\phi_\mu$ solves~\eqref{phi m}   we observe that for any $R>0$
\ant{
0 \ge \left(\mu^2-\frac{1}{4}\right)\int_0^R\phi_\mu g(Q) dr
 = -\int_0^R\phi_\mu^{\prime\prime} g(Q) dr+\int_0^R V \phi_\mu g(Q)dr.
}
Since
\begin{multline*}
	-\int_0^R\phi_\mu^{\prime\prime} g(Q) dr=-\phi_\mu^\prime(R) g(Q(R))+\int_0^R\phi_\mu^\prime\left(g(Q)\right)^\prime dr\\
	=-\phi_\mu^\prime(R)g(Q(R))+k\phi_\mu(R)\frac{g (Q(R))g^\prime(Q(R))}{\sinh R} -\int_0^R\phi_\mu\left(g(Q)\right)^{\prime\prime}dr,
\end{multline*}
it follows that
\ant{
0\geq -\phi_\mu^\prime(R)g(Q(R))+k\phi_\mu(R)\frac{g(Q(R))g^\prime(Q(R))}{\sinh R}+\int_0^R(V-U)\phi_\mu g(Q) dr.
}
Note that since $\la \le 1$, the second term above is non-negative. Using this along with the bound~\eqref{v-u} we have 
\ant{
 \frac{3}{4}  \int _0^R \phi_\mu g(Q)  \frac{1}{ \sinh^2r} \, dr \le \phi_\mu^\prime(R)g(Q(R))
 }
 We note that the left-hand side above is strictly positive and increasing in $R$ and hence can be bounded below by a fixed constant $\de>0$. This implies that 
 \ant{
 0 < \de  \le  \phi_\mu^\prime(R)g(Q(R)),   \, \quad  \forall \, R>0
 }
 However, the right-hand side above tends to zero as $R \to \infty$ which yields a contradiction for $R>0$ large enough, completing the proof. 
\end{proof}

\subsection{Proof of Theorem~\ref{S2 no eval}$(ii)$}  \label{s:k}
Next, we prove that in the case of $k$-equivariant wave maps, one can rule out eigenvalues and resonances for harmonic maps with images in  a region that is slightly larger than the northern hemisphere, \emph{which is independent of $k$}. In other words, we prove Theorem~\ref{S2 no eval}$(ii)$.

In the proof of Theorem~\ref{S2 no eval}$(i)$ we established the fact that for each $k \in \N$ and for each $\la \le 1$ the spectrum of the linearized operator $\Lline$ is purely absolutely continuous, with $\s(\Lline) = [1/4, \infty)$, where the threshold $\frac{1}{4}$ is neither an eigenvalue nor a resonance. 

We claim that for each $k$ we can find a number $\de(k) >0$ so that spectrum of $ \s(\Lline) = [1/4, \infty)$ for all $\la< 1+ \de(k)$ is purely absolutely continuous with no eigenvalue or resonance at the edge $1/4$. 
To see this one can study the nonzero solutions to $\Lline \phi_0^{\la} = \frac{1}{4} \phi_0^{\la}$ which satisfy  $\phi_{0}^{\la} \in L^2[0, c)$ for all $c>0$. We normalize so that $\phi_{0}^{\la} = r^{k+\frac{1}{2}} + o(r^{k+\frac{1}{2}})$. 
By Sturm's oscillation theory, having an eigenvalue $\mu^2 < 1/4$ is equivalent to such $\phi_{0}^\la$ having a zero (i.e.,  changing signs).  Thus, by Theorem~\ref{S2 no eval}$(i)$ and Lemma~\ref{LL ev 1}, for each fixed $k$ there exists a positive solution $\phi^1_0$ to $\LL_{V_{1, k}} \phi^1_0 = \frac{1}{4} \phi_0^{1}$ with $\phi_{0}^1 \in L^2[0, c)$ for all $c>0$, but $\phi^1_0 \not \in L^2[0, \infty)$. Moreover, such a positive solution $\phi_{0}^1$ has the property that we can find numbers $a , b \in\R$ with $b>0$  (this latter condition since $\phi_0^1$ cannot be a resonance) so that 
\ant{
\phi_{0}^1(r)  = a + b r +O(r e^{-2r}) \mas r \to \infty
}
Finally, by the continuity of $\phi^\la$ in $r$ and $\la$ at $\la =1$, we can find $\e(k)$ such that for all $\la< 1+ \e(k)$ there exist positive solutions $\phi_0^\la$ to $\Lline \phi_0^{\la} = \frac{1}{4} \phi_0^{\la}$ which satisfy  $ \phi^{\la} \in L^2[0, c)$ for all $c>0$, but $\phi^\la \not \in L^2[0, \infty)$. And moreover, for any such positive solution $\phi^\la \in L^{2}[0, c)$ we can find numbers $a, b$ with $b>0$ so that 
\ant{
\phi^\la(r)  = a + b r +O(r e^{-2r}) \mas r \to \infty
}
To understand the conclusions reached above in terms of the geometry of the image of the underlying harmonic maps $Q_{\la, k}$,  with $k$ and $\la$ both varying, it is more natural to consider the number $\la^k$ rather than $\la$ as the parameter which measures the angle corresponding to how far $Q_{\la, k}$ wraps around the sphere. Indeed we have 
\ant{
Q_{\la, k}( \infty) = 2 \arctan(\la^k).
}
With this in mind, setting $$\Theta:= \la^k$$ we note that we have proved that for each $k \in \N$, there exists a number $\de(k)>0$ so that for all $\Theta< 1 + \de(k)$ the spectrum $\s(\Lline) = [1/4, \infty)$ is purely absolutely continuous with no eigenvalue or resonance at the edge $1/4$; see Proposition~\ref{no neg spec} for the absence of eigenvalues below $0$ or embedded in $[1/4, \infty)$.

The goal now is to show that $\de = \de(k)$ can in fact be chosen independently of $k$. This requires an examination of the spectrum of $\Lline$ in the limit $k \to \infty$. To examine this behavior we find it convenient to consider the following change of variables. Set 
\EQ{
 \rho :=  \la^k \tanh^k(r/2) =  \Theta \tanh^k(r/2)
 }
Then $ r = r( \rho) = 2 \arctanh( ( \rho/ \Theta)^{\frac{1}{k}})$ and 
\ant{
 \frac{\p r}{ \p \rho} =    \frac{2}{k}\left[\left( \frac{\Te}{\rho} \right)^{\frac{1}{k}} - \left( \frac{\Te}{\rho} \right)^{-\frac{1}{k}} \right]^{-1} \frac{1}{\rho}
 }
 For convenience we define the function 
 \EQ{
 \om_{k, \Te}( \rho) :=  \frac{2}{k}\left[\left( \frac{\Te}{\rho} \right)^{\frac{1}{k}} - \left( \frac{\Te}{\rho} \right)^{-\frac{1}{k}} \right]^{-1}
 }
 Thus we have 
 \EQ{
  dr = \om_{k,\Te} (\rho)\frac{d \rho}{\rho} , \qquad  \frac{ \p}{\p r} = \left(\frac{ \p r}{ \p \rho}\right)^{-1}  \frac{\p}{\p \rho}  =   (\om_{k,\Te} (\rho))^{-1} \rho  \frac{\p}{ \p \rho} 
  }
  Recall that we are studying the operator $\Lline$ which can be expressed as follows: 
  \EQ{
  \Lline &=  - \frac{\p^2}{\p r^2} + \frac{1}{4} +  \frac{k^2-\frac{1}{4}}{\sinh^2 r} + V_{\la, k}  \\
  & =  - \frac{\p^2}{\p r^2} + \frac{1}{4} - \frac{1}{4 \sinh^2 r} + k^2 \frac{ \cos 2Q_{\la, k}}{  \sinh^2 r} 
  }
  We note that in the new variable $\rho$, we have 
  \EQ{
 \sinh r = k \om_{k, \Te}(\rho), \qquad  \cos 2 Q_{\la, k}(r) =  \frac{1 - 6  \rho^2 + \rho^4}{ (1+ \rho^2)^2}
  }
Therefore,  we obtain the renormalized operator $\LL_{k, \Te}$ in the $\rho$-variable defined by 
\EQ{
\LL_{k, \Te} \fy = - \om^{-1}_{k, \Te}   \rho \p_{\rho} \left(  \om^{-1}_{k, \Te} \rho \p_ \rho \fy\right) + \frac{1}{4} \fy  - \frac{1}{4k^2}  \om^{-2}_{k, \Te}   \fy +  \om^{-2}_{k, \Te} \frac{1 - 6  \rho^2 + \rho^4}{ (1+ \rho^2)^2} \fy
} 
so that if we set $\fy( \rho) =   \phi( r)$, then $\Lline \phi (r) =  \LL_{k, \Te} \fy ( \rho)$. The convenience in this change of variables is that it is easy to understand the limiting behavior in $k$. Indeed, we have the following formula. 


\begin{lem}\label{k limit} For each fixed $\rho < \Te$, we have the following point-wise-in-$k$ limit. 
\ant{
 \lim_{k \to \infty} \om_{k, \Te}( \rho)= \lim_{k\to\infty}\left[\left(\frac{k}{2}\right)\left(\left(\frac{\Theta}{\rrho}\right)^{\frac{1}{k}}-\left(\frac{\Theta}{\rrho}\right)^{-\frac{1}{k}}\right)\right]^{-1}=\log^{-1}\left(\frac{\Theta}{\rrho}\right).
}
In fact the limit is monotone. Indeed,  for all $\rrho<\Theta$ and $k\geq1$ we have
\EQ{ \label{om inf} 
 \om_{k, \Te}( \rho)\leq  \om_{\infty, \Te}( \rho) := \log^{-1}\left(\frac{\Theta}{\rrho}\right),
}

\end{lem}
\begin{proof}
For the second statement let $x=\frac{\Theta}{\rrho}\geq1,$ and $\alpha=\frac{1}{k}.$ Then
\ant{
\frac{x^\alpha-x^{-\alpha}}{2\alpha}=\frac{1}{2\alpha}\int_{-\alpha}^\alpha\frac{d}{da}x^ada=\frac{\log x}{2\alpha}\int_{-\alpha}^\alpha x^ada\geq\left(\log x\right) x^{\frac{1}{2\alpha}\int_{-\alpha}^\alpha ada}=\log x,
} 
where we have used Jensen's inequality for convex functions. The desired inequality now follows by raising both sides to the power of $-1.$ The limit can be computed for instance using l'H\^opital's rule. 
\end{proof}

By Lemma~\ref{k limit} we have the following formal limit 
\ant{
 \LL_{k, \Te} \to  \LL_{ \infty, \Te} \mas k \to \infty
 }
where 
\ant{
\LL_{\infty, \Te}  \fy &:= - \rho \log\left(\frac{\Theta}{\rrho}\right) \p_{\rho} \left( \rho \log\left(\frac{\Theta}{\rrho}\right) \p_ \rho \fy \right) + \frac{1}{4} \fy +  \log^2\left(\frac{\Theta}{\rrho}\right)  \frac{1 - 6  \rho^2 + \rho^4}{ (1+ \rho^2)^2} \fy \\
& = -\rho \,  \om^{-1}_{\infty, \Te} \p_{\rho} \left( \rho\, \om^{-1}_{\infty, \Te} \p_ \rho \fy \right) + \frac{1}{4} \fy +  \om^{-2}_{\infty, \Te} \frac{1 - 6  \rho^2 + \rho^4}{ (1+ \rho^2)^2} \fy
}
The idea is to  first study solutions $\fy = \fy (\rho) $ to 
\EQ{ \label{Linf} 
\LL_{\infty, \Te} \fy = \mu^2 \fy
} and then use this information to understand spectral properties of $\LL_{k, \Te}$ for $k$ large. First, we note a few elementary properties of solutions to~\eqref{Linf}. 

\begin{lem} \label{l:Linf0} 
Let $\fy$ be any solution to~\eqref{Linf} with $\varphi_{\mu} \in L^{2}([0, c), \om_{\infty, \Te} \,  \rho^{-1} \, d \rho)$ for some $c \in (0, \Tht]$ and $\mu^{2} \leq 1/4$. Then there exists a number $a \in \bbR$ such that
\begin{align} 
 \fy( \rho)  = & a \rho \log^{-\frac{1}{2}} \left( \frac{\Tht}{\rho} \right) + o \left(\rho \log^{-\frac{1}{2}} \left( \frac{\Tht}{\rho} \right) \right) \mas  \rho \to 0, \label{eq:Linf0:zeroth-order} \\
 \fy'( \rho) = & a \left[ \log^{-\frac{1}{2}} \left( \frac{\Tht}{\rho} \right) + \frac{1}{2} \log^{-\frac{3}{2}} \left( \frac{\Tht}{\rho} \right) \right] + o \left( \log^{-\frac{1}{2}} \left( \frac{\Tht}{\rho} \right) \right) \mas \rho \to 0. \label{eq:Linf0:first-order}
\end{align}
\end{lem} 
\begin{proof} 
Consider the linear operator
\begin{equation*}
	L_{\infty, \Tht} \varphi := - \frac{1}{\rho \log \left( \frac{\Tht}{\rho} \right)} \rd_{\rho} ( \rho \log \left( \frac{\Tht}{\rho} \right) \rd_{\rho} \varphi) + \frac{1}{4 \rho^{2} \log^{2} \left( \frac{\Tht}{\rho} \right)} \varphi + \frac{1}{\rho^{2}} \varphi
\end{equation*}
with a fundamental system $\set{\rho \log^{-\frac{1}{2}} \left( \frac{\Tht}{\rho} \right), \rho^{-1} \log^{-\frac{1}{2}} \left( \frac{\Tht}{\rho} \right)}$,
and the corresponding Green's function
\begin{align*}
	G(\rho, \tau) 
	= & - \frac{1}{2} \rho^{-1} \log^{-\frac{1}{2}} (\Tht/\rho) \, \tau^{2} \log^{\frac{1}{2}}(\Tht/ \tau)
		+ \frac{1}{2} \rho \log^{-\frac{1}{2}} (\Tht/\rho) \, \log^{\frac{1}{2}} (\Tht/\tau).
\end{align*}
The fundamental system for $L_{\infty, \Tht}$ can be used to approximate solutions to $(\calL_{\infty, \Tht} - \mu^{2}) \varphi = 0$ by the relation
\begin{equation*}
	L_{\infty, \Tht} - \frac{1}{\rho^{2} \log^{2} \left( \frac{\Tht}{\rho} \right)} (\calL_{\infty, \Tht} - \mu^{2}) = \frac{\mu^{2}}{\rho^{2} \log^{2}\left( \frac{\Tht}{\rho} \right)} + O(1) 
\end{equation*}
Indeed, using the variation of constants formula for $L_{\infty, \Tht}$ and using Picard iteration starting from $\rho \log^{-\frac{1}{2}} (\Tht/ \rho)$ and $\rho^{-1} \log^{-\frac{1}{2}} (\Tht/ \rho)$ on an interval of the form $(0, c)$ for sufficiently small $c > 0$, we obtain two solutions $\varphi_{\mu; 1}, \varphi_{\mu; -1}$ to $(\calL_{\infty, \Tht} - \mu^{2}) \varphi = 0$ with the asymptotics
\begin{align*}
	\varphi_{\mu; 1}(\rho) =& \rho \log^{-\frac{1}{2}} \left( \frac{\Tht}{\rho} \right) + o \left( \rho \log^{-\frac{1}{2}} \left( \frac{\Tht}{\rho} \right) \right) \mas \rho \to 0, \\
	\varphi_{\mu; -1}(\rho) =& \rho^{-1} \log^{-\frac{1}{2}} \left( \frac{\Tht}{\rho} \right) + o \left( \rho^{-1} \log^{-\frac{1}{2}} \left( \frac{\Tht}{\rho} \right) \right) \mas \rho \to 0.
\end{align*}
These solutions are clearly linearly independent, and hence they span the set of all solutions to $(\calL_{\infty, \Tht} - \mu^{2}) \varphi = 0$. Since $\varphi_{\mu, -1} \not \in L^{2}([0, c); \omg_{\infty, \Tht} \rho^{-1} d \rho)$, the asymptotics \eqref{eq:Linf0:zeroth-order} follows. The asymptotics \eqref{eq:Linf0:first-order} follows by differentiating the variation of constants formula to obtain  asymptotics for $\varphi'_{\mu; 1}$. \qedhere
\end{proof}

 
 \begin{lem} \label{l:Linf1}  
 Suppose $\fy$ is a solution to~\eqref{Linf} with $\mu^2  = \frac{1}{4}$. Then there exist numbers $a, b \in \R$ so that 
  \EQ{
 \fy( \rho) = a  - b \log \log( \Te/ \rho)  + O( \abs{\log \log( \Te/ \rho)} \log^2( \Te/ \rho)) \mas  \rho \to  \Te
 } 
 If $\fy$ is a resonance or an eigenvalue with $\mu^2 = 1/4$, then $b=0$ and we have 
  \EQ{
 \fy'( \rho) = O(\abs{\log \log ( \Te/ \rho)} \log( \Te/ \rho)) \mas  \rho \to \Te
 }
 Next, suppose $\fy$ is solution to~\eqref{Linf} with $\mu^2 < 1/4$. Then there exist numbers $c_+, c_- \in \R$ so that as $\rho  \to \Te$ we have 
 \EQ{ \label{c34}
\begin{aligned}
 \phi(\rho) =&  c_- \log^{-\sqrt{1/4 - \mu^2}}( \Te/ \rho) + c_+ \log^{\sqrt{ 1/4- \mu^2}}( \Te/ \rho) \\
 	& + O(  \log^{-\sqrt{ 1/4- \mu^2}}( \Te/ \rho) \log^{2} (\Tht/\rho)) 
\end{aligned} }  
 If $\mu^2< 1/4$ is an eigenvalue then we have $c_- = 0$ above. 
\end{lem} 
 \begin{proof}
 Here it is convenient to make the change of variables 
 \EQ{ \label{log log} 
 s := - \log \log( \Te/ \rho),  \,  \quad   \frac{\p}{ \p s} =  \rho \log ( \Te / \rho) \frac{\p}{ \p \rho}
  }
 and write $\psi( s) = \fy( \rho)$. In these new variables~\eqref{Linf}  with $\mu^2  \le  1/4$ becomes 
 \EQ{ \label{s eq} 
   - \psi'' + e^{-2s} \frac{1 - 6  \rho(s)^2 + \rho(s) ^4}{ (1+ \rho(s)^2)^2}  \psi = ( \mu^2 - 1/4) \psi, 
   }
   where $ \rho(s) =  \Te e^{- e^{-s}}$. As in Lemma~\ref{LL ev 1} it now follows from the variation of parameters formula that 
   \ant{
   &\psi(s)  = a + b s + O(s e^{-2s}) \mas s \to \infty \mif \mu^2 = 1/4 \\
   &\psi(s)  = c_- e^{s \sqrt{1/4-\mu^{2}}} + c_+ e^{- s \sqrt{1/4-\mu^{2}}} + O( e^{s\sqrt{1/4-\mu^{2}}} e^{-2s} )  \mas s \to \infty \mif  \mu^2 <1/4
   }
   Undoing the change of variables above yields the lemma. 
 \end{proof} 

Next we prove by a comparison argument that for $\Te \le 1$ the spectrum of $\LL_{\infty, \Te}$ is purely absolutely continuous and is given by $[1/4, \infty)$ with $1/4$ neither an eigenvalue nor a resonance. 
\begin{lem}\label{l:Lreg} 
If $\Te \le 1$ then the spectrum $\s( \LL_{\infty, \Te}) = [1/4, \infty)$ is purely absolutely continuous and the threshold $1/4$ is neither an eigenvalue nor a resonance. 
\end{lem} 
\begin{proof}
The proof proceeds via a comparison argument in the same spirit as the proof of Theorem~\ref{S2 no eval}$(i)$. We note that the space $L^2 $ is defined as $L^2( [0, \Te);  \om_{\infty, \te} \rho^{-1} \, d \rho )$. 

 Let $\mu \in \C$ with $\mu^2 \le \frac{1}{4}$. Suppose that  $\fy_\mu$ is a solution to
\EQ{\label{phi mu}
\LL_{\infty, \Te} \fy_\mu = \mu^2 \fy_{\mu}.
}
If  $\mu^2 \le 1/4$ is an eigenvalue, we assume  that it is the smallest eigenvalue, and by a variational principle, we can further assume that the corresponding eigenfunction $\fy_{\mu} \in L^2$ is unique, (i.e., $\mu^2$ is simple) and strictly positive. If $\mu^2 = \frac{1}{4}$ and is not an eigenvalue, we  assume that $\fy_{\mu}$ is a non-negative threshold resonance. 

Next, we introduce the function 
\EQ{ \label{Phi}
 \Phi( \rho)  = \frac{ \rho}{ 1 + \rho^2}
 }
We note that $\Phi$ solves the equation 
\EQ{ \label{Phi eq} 
 (\LL_{\infty, \Te} - 1/4) \Phi =  \log( \Te/ \rho) \frac{ 1- \rho^2}{ 1+ \rho^2} \Phi
 }
 As an aside, we remark that the motivation for introducing the function $\Phi( \rho)$ comes from the role that the function $\sin Q_{\la, k}$ played in the proof of Theorem~\ref{S2 no eval}$(i)$ and the fact that we have the point-wise  limit 
 \ant{
 \lim_{k \to \infty}   \sin(Q_{\la, k} ( \rho)) =  \Phi( \rho)
 }
in the $\rho$ coordinate. 

Let $ \Te \le 1$ and let $0 < \e < R < \Te \le 1$. Integrating by parts we obtain, 
\ant{
  0 &\ge \left( \mu^2 - 1/4 \right) \int_\e^R \fy_\mu( \rho) \,  \Phi( \rho)  \,  \om_{\infty, \Te}( \rho) \rho^{-1} \, d \rho   \\
  & =  \int_\e^R ((\LL_{\infty, \Te}- 1/4) \fy_\mu)( \rho) \,  \Phi( \rho)  \,  \om_{\infty, \Te}( \rho) \rho^{-1} \, d \rho \\
  & =  \int_{\e}^R  \phi_ \mu( \rho) ((\LL_{\infty, \Te}- 1/4) \Phi)( \rho)  \, \om_{\infty, \Te}( \rho) \rho^{-1} \, d \rho  \\
  & \quad + \om^{-1}(R) R \,  \Phi'( R) \fy_\mu(R) -  \om^{-1}( \e) \e \,  \Phi'( \e) \fy_\mu(\e) \\
  & \quad  -   \om^{-1}(R) R \,  \Phi( R) \fy_\mu'(R) +  \om^{-1}( \e) \e \,  \Phi( \e) \fy_\mu'(\e)
  }
  Using~\eqref{Phi}~\eqref{Phi eq},  and the definition of $\om_{ \infty, \Te}$ the above becomes 
  \ant{ 
  \int_{\e}^R  \frac{1 - \rho^2}{(1+ \rho^2)^2}  \phi_ \mu( \rho)  \, d \rho &\le  -  \om^{-1}(R) R \,  \Phi'( R) \fy_\mu(R) +  \om^{-1}( \e) \e \,  \Phi'( \e) \fy_\mu(\e)  \\
  & \quad +  \om^{-1}(R) R \,  \Phi( R) \fy_\mu'(R)  -   \om^{-1}( \e) \e \,  \Phi( \e) \fy_\mu'(\e)
  }
  Note that since we are assuming that $\Te \le 1$ and we always have $\rho < \Te \le 1$ we know that the left-hand side above strictly positive and increasing as $ \e  \to 0$ and as $R \to \Te$. Hence we can bound the left-hand side below by a fixed constant $\de>0$ which gives 
  \ant{
   0  < \de  &\le   -  \om^{-1}(R) R \,  \Phi'( R) \fy_\mu(R) +  \om^{-1}( \e) \e \,  \Phi'( \e) \fy_\mu(\e)  \\
  & \quad +  \om^{-1}(R) R \,  \Phi( R) \fy_\mu'(R)  -   \om^{-1}( \e) \e \,  \Phi( \e) \fy_\mu'(\e)
  }
  Finally, we note that the terms involving $\e$ on the right-hand side above both tend to $0$ as $ \e \to 0$ due to the $L^2([0, c);  \om_{ \infty, \Te} \, \rho^{-1}  \, d \rho)$ condition on $\phi_\mu$ and Lemma~\ref{l:Linf0}. Also, the terms involving $R$ on the right-hand side both tend to $0$ as $ R \to \Te$. If $\mu^2< 1/4$ is an eigenvalue this is due to the fact that  $ \fy_\mu \in L^2([0,  \Te);  \om_{ \infty, \Te} \, \rho^{-1}  \, d \rho)$. If $\mu^2 = 1/4$ is either an eigenvalue or a resonance, then this is due to Lemma~\ref{l:Linf1}. This gives a contradiction by taking $\e \to 0$ and $R \to \Te$ above. 
  
  Finally, we note that it remains to rule out embedded eigenvalues in the continuous spectrum. This is straightforward and follows from the same argument used below in the proof of Proposition~\ref{no neg spec}$(ii)$, but here using the change of variables~\eqref{log log} and the reduction to an equation of the form~\eqref{s eq} with $\mu^2> 1/4$. We omit the details. 
 \end{proof} 
 
 Finally we are ready to complete the proof of Theorem~\ref{S2 no eval}$(ii)$. 
 
 \begin{proof}[Proof of Theorem~\ref{S2 no eval}$(ii)$]
Arguing as in the beginning of Section~\ref{s:k} where we showed that the spectrum of $\LL_{\la, k}$ is given by $[1/4 , \infty)$ with no eigenvalues or resonances at $1/4$ for all $\la < 1 + \de(k)$, we observe that there exists $\Te_0 >1$ so that $\LL_{\infty, \Te}$ has spectrum $\s( \LL_{\infty, \Te}) = [1/4,  \infty)$ which is purely absolutely continuous with no eigenvalues or resonance at the edge $1/4$ for all $\Te <  \Te_0$. This is simply a manifestation of the fact that not having a eigenvalue $\mu^2 \le 1/4$ or a resonance at $1/4$ is an open condition in $\Te$. 

Next, we would like to exploit the formal convergence $\LL_{k, \Te} \to  \LL_{\infty, \Te}$ to prove the existence of the number $\Te_*>1$ as in the statement of Theorem~\ref{S2 no eval}$(ii)$. 

Suppose that Theorem~\ref{S2 no eval} is false. Then we can find a sequence $k_n \to \infty$,  a decreasing sequence of numbers  $\Te_{n}  \searrow 1$, a sequence $\mu^2_n \le 1/4$, and a sequence of  smooth \emph{bounded, positive solutions} $\fy_n \in L^{\infty}([0,\Te_n))$  to 
\ant{
\LL_{k_n, \Te_{n}} \fy_n = \mu_n^2 \fy_n. 
}
\begin{rem}\label{r:bounded} 
We remark here that $\LL_{k, \Te}$ has an eigenvalue $\mu_0^2\le 1/4$ \emph{or} a resonance at $1/4$ if and only if there exists a bounded positive solution $\fy$ to $\LL_{k, \Te} \phi  = \mu_1^2 \phi$ where we possibly have $\mu_1 < \mu_0$. Indeed, if $\mu^2_0$ is an eigenvalue we can take $\mu_1^2$ to be the smallest eigenvalue, and by a variational principle we can assume that $\mu_1^2$ is simple, and the corresponding eigenfunction $\fy$ is strictly positive. This eigenfuction is bounded by Lemma~\ref{LL ev 0} and Lemma~\ref{LL ev 1}. If $\mu_0^2 = 1/4$ is a resonance, then Sturm's oscillation theorem together with Lemma~\ref{LL ev 0} and Lemma~\ref{LL ev 1} ensure that it is strictly positive and bounded. For these reasons, we can guarantee that our sequence $\fy_n$ above can be taken to be bounded positive functions. Note that that same remark holds for the operator $\LL_{\infty, \Te}$ using Lemma~\ref{l:Linf0} and Lemma~\ref{l:Linf1}. 
\end{rem} 
 
 Returning to the proof, we can renormalize the sequence $\phi_n$ in $L^{\infty}$ so that $\phi_n( \rho) \le 1$ for all $\rho \in [0, \Te_n)$ and there is a sequence of points $\rho_n \in [0, \Te_n]$ so that $\phi_n( \rho_n)  =1$ for each $n$. Passing to a subsequence, we can assume that $\rho_n \to \rho_* \in [0, 1]$ and moreover that our eigenvalues (or resonances) $\mu_n^2 \to \mu_\infty^2  \in [0, 1/4]$ (note that the lower bound on the $\mu_\infty^2$ can be ensured due to the fact that $\LL_{k, \Te}$ has no negative spectrum, see Proposition~\ref{no neg spec} below). The outline for the remainder of the proof is as follows
 \begin{itemize}
 \item[Step $1$.] Let $J$ be any compact subset $J \Subset (0, 1)$. We prove that the sequence $\{\fy_n\}$ is equicontinuous on $J$. We can then find a subsequence, and a bounded continuous function $\fy_{\infty}$ defined on $(0, 1)$ so that $\fy_n \to \fy_{\infty}$ uniformly on  each $J \Subset (0, 1)$. 
 \item[Step $2$.] We show that $\fy_{\infty}$ solves $\LL_{\infty, 1} \fy_{\infty}  = \mu_\infty^2 \fy_{\infty}$. 
 \item[Step $3$.] We show that $\fy_{\infty}$ is not identically $ \equiv 0$, which together with Step $2$ and Remark~\ref{r:bounded} proves that $\fy_{\infty}$ is either and eigenvalue or threshold resonance for $\LL_{\infty, 1}$, which is impossible  by  Lemma~\ref{l:Lreg}. This is our contradiction. 
 \end{itemize}

First we prove Step $1$. Fix $J \Subset (0, 1)$ and let $\chi$ be a smooth non-negative function  compactly supported function on $(0, 1$) that is identically $=1$ on  $J$. We multiply the equation for $\fy_{n}$, i.e., $\LL_{k_n, \Te_n} \fy_n = \mu^2_n \fy_n$,  by $\fy_{n}\chi$ and integrate by parts (with the measure $\omega_{k_n, \Theta_n}(\rho) \rho^{-1}\, d\rho$) to get
\begin{align*}
& \hskip-1em
\int_0^{1}|\partial_\rho\fy_{n}(\rho)|^2\chi(\rho)\omega_{k_n, \Te_n}^{-1} \rho d\rho \\ 
= &\frac{1}{2}\int_0^{1}\partial_\rho \big(\omega_{k_n, \Te_n}^{-1}\rho)\partial_\rho\chi(\rho) \big) \fy_{n}^2(\rho)d\rho
+(\mu_n^2- \frac{1}{4})\int_0^{1}\fy_{n}^2(\rho)\chi(\rho)\omega_{k_n, \Te_n}(\rho) \rho^{-1}d\rho \\ 
  +  &\frac{1}{4k^2} \int_0^1 \fy_{n}^2(\rho)\chi(\rho) \omega_{k_n, \Te_n}^{-1}(\rho) \, \rho^{-1}d \rho    - \int_0^{1} \frac{1- 6 \rho^2 + \rho^4}{(1+ \rho^2)^2} \fy_{n}^2(\rho)\chi( \rho) \omega_{k_n, \Te_n}^{-1}(\rho) \, \rho^{-1}d \rho.
\end{align*}
Now note that we can find constants $c= c(J), C= C(J)$ so that  $$0 < c \le \abs{\omega_{k_n, \Theta_n}(\rho)} \le C < \infty$$ on $J$ uniformly in $n$.  Since $|V_{k_n}|,~|\partial_\rho\omega^{-1}_{k_n}(\Theta_n,\rho)|,$ and $\abs{\fy_{k_n}}$ are also uniformly bounded on $J$,  we conclude that
\ant{
\int_{J}|\partial_\rho \fy_{n}(\rho)|^2d\rho\lesssim_J 1, 
}
uniformly in $n.$ It follows from an application of the Fundamental Theorem of Calculus and  the Cauchy-Schwarz inequality that the sequence $\fy_{n}$ is equicontinuous on $J$ for any fixed  $J  \Subset (0, 1)$. Moreover, the sequence is uniformly bounded on $(0, 1)$.  It follows from the Arzela-Ascoli that after passing to a subsequence $\fy_n$ converges to a bounded continuous function $\fy_{\infty}$ defined on $(0,1)$.  Moreover, the convergence is uniform on every compact subinterval of $(0,1)$. This proves Step $1$. 

To prove Step $2$, we let $\chi \in C^{\infty}_0(0,1)$ be a test function. Then we have 
\ant{
\int_0^1 &\LL_{\infty, 1} \chi \fy_{\infty} \om_{ \infty, 1}  \frac{d \rho}{ \rho} = \int_0^1 \LL_{\infty, 1} \chi ( \fy_{\infty}- \fy_n) \om_{ \infty, 1}  \frac{d \rho}{ \rho}  + \int_0^1 \LL_{\infty, 1} \chi \fy_{n} \om_{ \infty, 1}  \frac{d \rho}{ \rho}  \\
 &= I + \int_0^1 (\LL_{\infty, 1} \chi) \fy_{n} (\om_{ \infty, 1} - \om_{k_n, \Te_n})  \frac{d \rho}{ \rho}  +  \int_0^1 (\LL_{\infty, 1} \chi) \fy_{n} \om_{ k_n, \Te_n}  \frac{d \rho}{ \rho}  \\
 & = I + II + \int_0^1 ((\LL_{\infty, 1} - \LL_{k_n, \Te_n}) \chi) \fy_{n} \om_{ k_n, \Te_n}  \frac{d \rho}{ \rho}  + \int_0^1 (\LL_{k_n, \Te_n} \chi) \fy_{n} \om_{ k_n, \Te_n}  \frac{d \rho}{ \rho}   \\
 & = I + II + III +  \mu_n^2 \int_0^1 \chi  \fy_{n} \om_{ k_n, \Te_n}  \frac{d \rho}{ \rho}  
}
where due to the support properties of $\chi$ we have 
\ant{
&I:= \int_0^1 \LL_{\infty, 1} \chi ( \fy_{\infty}- \fy_n) \om_{ \infty, 1}  \frac{d \rho}{ \rho}  = o_n(1) \mas n \to \infty\\
&II:=  \int_0^1 (\LL_{\infty, 1} \chi) \fy_{n} (\om_{ \infty, 1} - \om_{k_n, \Te_n})  \frac{d \rho}{ \rho}  = o_n(1) \mas n \to \infty\\
&III:=  \int_0^1 (( \LL_{\infty, 1} - \LL_{k_n, \Te_n}) \chi) \fy_{n} \om_{ k_n, \Te_n}  \frac{d \rho}{ \rho}  = o_n(1) \mas n \to \infty
}
Finally, note that 
\ant{
\mu_n^2 \int_0^1& \chi  \fy_{n} \om_{ k_n, \Te_n}  \frac{d \rho}{ \rho}   =  ( \mu_n^2 - \mu_{\infty}^2)  \int_0^1 \chi  \fy_{n} \om_{ k_n, \Te_n}  \frac{d \rho}{ \rho}  + \mu_{\infty}^2 \int_0^1 \chi  \fy_{n} \om_{ k_n, \Te_n}  \frac{d \rho}{ \rho}  \\
& = o_n(1) + \mu_{\infty}^2 \int_0^1 \chi  (\fy_{n} \om_{ k_n, \Te_n} -  \fy_{\infty} \om_{\infty, 1}) \frac{d \rho}{ \rho}  + \mu_{\infty}^2 \int_0^1 \chi  \fy_{\infty} \om_{\infty, 1}  \frac{d \rho}{ \rho}   \\
& = o_n(1)+ \mu_{\infty}^2 \int_0^1 \chi  \fy_{\infty} \om_{\infty, 1}  \frac{d \rho}{ \rho}
}
Hence, 
\ant{
\int_0^1 &\LL_{\infty, 1} \chi \fy_{\infty} \om_{ \infty, 1}  \frac{d \rho}{ \rho} =  \mu_{\infty}^2 \int_0^1 \chi  \fy_{\infty} \om_{\infty, 1}  \frac{d \rho}{ \rho} + o_n(1) \mas n \to \infty
}
which proves that $\fy_{\infty}$ is a weak solution of $\LL_{\infty, 1} \fy_{\infty} = \mu_{\infty}^2 \fy_{\infty}$. Finally, since all the coefficients of $\LL_{\infty, 1} $ are bounded on any compact subinterval of $(0, 1)$, we can conclude that $\fy_{\infty}$ is in fact a strong solution. 

Lastly, we prove Step $3$, i.e., that $\fy_{\infty} \not \equiv 0$. Recall that we have chosen points $\rho_n \to \rho_* \in [0, 1]$ so that $\fy_n( \rho_n) = 1  =  \sup_{\rho \in (0, \Te_n]}   \fy_n(\rho)$. If the supremum is achieved at the endpoint $\Te_n$ then $\fy_n( \rho_n) = \lim_{\rho \to \Te_n} \fy( \rho)$ is interpreted as a limit. Note that if $\rho_* \in (0, 1)$, then we are done, since then we can find a compact set $J \Subset (0, 1)$ so that $\rho_* \in J$, and so that $\rho_n \in J$ for all $n$ large. Since we know that $\fy_n \to \fy_{\infty}$ uniformly on $J$, and since $\fy_n( \rho_n)  = 1$, we can conclude that $\fy \not \equiv 0$. 

We are thus left with the remaining possibilities $\rho_*  = 0$ and $\rho_* = 1$. We can easily rule out the first scenario. Indeed, recall that $\fy_n$ satisfies 
\ant{
 \om^{-1}_{k_n, \Te_n}   \rho \p_{\rho} \left(  \om^{-1}_{k, \Te_n} \rho \p_ \rho \fy_n\right)  =  \left( 1/4 - \mu_n^2 \right) \fy_n  +     \om^{-2}_{k_n, \Te_n}  \left(\frac{1 - 6  \rho^2 + \rho^4}{ (1+ \rho^2)^2} - \frac{1}{4k_n^2} \right) \fy_n
}
Multiplying both sides by $\om_{k_n, \Te_n} \rho^{-1}$,  integrating from $0$ to $\rho$, and then multiplying again by $\om_{k_n, \Te_n} \rho^{-1}$ yields, 
\begin{multline*}
\p_ \rho \fy_n( \rho) = \left( 1/4 - \mu_n^2 \right) \frac{ \om_{k_n, \Te_n}}{ \rho}\int_0^{\rho} \fy_n  \om_{k_n, \Te_n} \frac{d \tau}{ \tau}  \\ +  \frac{\om_{k_n, \Te_n}}{ \rho}  \int_0^{\rho}  \om^{-2}_{k_n, \Te_n}  \left(\frac{1 - 6  \tau^2 + \tau^4}{ (1+ \tau^2)^2} - \frac{1}{4k_n^2} \right) \fy_n \om_{k_n, \Te_n}\frac{d \tau}{ \tau}
\end{multline*}
Now note that the right-hand side above is strictly positive on an interval $(0, \e]$ where $\e>0$ can be chosen \emph{independently of $n$}. Indeed, the right-hand side is positive as long as 
\ant{
\frac{1 - 6  \tau^2 + \tau^4}{ (1+ \tau^2)^2} - \frac{1}{4k_n^2}  > 0
}
and we can thus take, say $\e = 1/10$. This means that $\phi_n$ is strictly increasing on the interval $[0, \e]$ for all $n$ and hence we can ensure that $\rho_n \ge \e$ for each $n$. Hence $\rho_* \ge \e$. 

Now suppose that $  \rho_n \to \rho_*  = 1$. Proceeding as above, but this time integrating from $\rho$ to $\Te_n$ we have 
\begin{multline*}
\p_ \rho \fy_n( \rho) =  - \left( 1/4 - \mu_n^2 \right) \frac{ \om_{k_n, \Te_n}}{ \rho}\int_{\rho}^{\Te_n} \fy_n  \om_{k_n, \Te_n} \frac{d \tau}{ \tau}  \\  -   \frac{\om_{k_n, \Te_n}}{ \rho}  \int_{\rho}^{\Te_n}  \om^{-2}_{k_n, \Te_n}  \left(\frac{1 - 6  \tau^2 + \tau^4}{ (1+ \tau^2)^2} - \frac{1}{4k_n^2} \right) \fy_n \om_{k_n, \Te_n}\frac{d \tau}{ \tau}
\end{multline*}
Since the first term  on the right-hand side above is nonpositive we have 
\ant{
\p_ \rho \fy_n( \rho)  \le -   \frac{\om_{k_n, \Te_n}}{ \rho}  \int_{\rho}^{\Te_n}  \om^{-2}_{k_n, \Te_n}  \left(\frac{1 - 6  \tau^2 + \tau^4}{ (1+ \tau^2)^2} - \frac{1}{4k_n^2} \right) \fy_n \om_{k_n, \Te_n}\frac{d \tau}{ \tau}
}
Next, we can choose $\rho$ close enough to $1$, e.g., take $\rho \ge 3/4$, and $\Te_n$ close enough to $1$, so that 
\ant{
\frac{1 - 6  \tau^2 + \tau^4}{ (1+ \tau^2)^2} - \frac{1}{4k_n^2} \le 0  \mand  \abs{\frac{1 - 6  \tau^2 + \tau^4}{ (1+ \tau^2)^2} - \frac{1}{4k_n^2}}  \le 2 \mfor  \tau  \in [3/4, \Te_n]
}
Using the above along with the fact that $ \fy_n \le 1$  and Lemma~\ref{k limit} we have for all $\rho \in [3/4,   \Te_n]$ that 
\ant{
\p_ \rho \fy_n( \rho)  \le  \frac{C_1}{ \log( \Te_n/ \rho)} \int_ \rho^{\Te_n}  \om^{-1}_{k_n, \Te_n}( \tau) \,  d \tau
}
Note also for $\tau  \ge 3/4$ and $\Te_n$ close enough to $1$ we can find a uniform in $n$ constant $C_2>0$ so that $\om^{-1}_{k_n, \Te_n}( \tau)  \le C \log( \Te_n/ \tau)$. Hence, 
\EQ{ \label{dfy}
\p_ \rho \fy_n( \rho)  \le  \frac{C_1C_2}{ \log( \Te_n/ \rho)} \int_ \rho^{\Te_n}   \log( \Te_n/ \tau) \, d \tau \le C' \abs{ \Te_n - \rho}
}
Hence we can find $R<1$ close enough to $1$ and $N$ large enough so we have $\rho_n \in [R, 1]$  for all $n \ge N$, and in addition by~\eqref{dfy} we can guarantee that  for $ \rho \in [R, 1]$ we have 
\ant{
 \fy_n( \rho) \ge 1/2, \quad \forall  \rho \in [R, 1].
 }
Hence $\fy_{\infty}( \rho)  \ge 1/2$  for $ \rho \in [R, 1)$ and therefore is not identically $0$. This completes the proof  of Step $3$ and therefore also of Theorem~\ref{S2 no eval}$(ii)$.
\end{proof}
\subsection{No negative spectrum and no embedded eigenvalues} \label{s:noneg}

Before we consider the existence of gap eigenvalues in the next section, we first show that any eigenvalue of the operator  $\calL_{\VV}$, with $\VV=V_{\la,k}$ or $W_\la,$ must occur in the spectral gap $(0, 1/4)$. The following proposition holds for all $\lmb \in [0, \infty)$ and $k \ge 2$. 

\begin{prop} \label{no neg spec}
With $\VV=V_{\la,k}$ or $W_\la,$ the following statements concerning $\calL_{\VV}$ hold.
\begin{enumerate}[(i)]
\item For every $\lmb \geq 0$, the spectrum of $\calL_{\VV}$ does not contain any non-positive reals, i.e.,
\begin{equation*}
	\sgm(\calL_{\VV}) \cap (-\infty, 0] = \emptyset.
\end{equation*}
\item There does not exist any eigenvalue in $[\frac{1}{4}, \infty)$.
\end{enumerate}
\end{prop}

 Before proceeding with the proof we introduce two additional ingredients. 
 We define $\z_0^{\la}$ to 
 be the unique $L^2_{\loc}$ solution to $\Lline \zeta = 0$, which is obtained by differentiating $Q_{\la, k}( \rho)$ with respect to~$\la$. Indeed, we have 
\EQ{ \label{zdef}
\zeta_0^{\la}(r):=\sinh^{\frac{1}{2}}r\, \p_{\la}\Qu(r)=\frac{2k\lambda^{k-1}\tanh^{k}(r/2)\sinh^{1/2}r}{1+\lambda^{2k}\tanh^{2k}(r/2)}. 
}
We also require the analog of $\zeta_0^\la$ for $\LL_{W_\la}$, which is defined as follows:
\EQ{ \label{edef}
\eta_0^{\la}(r):=\sinh^{\frac{1}{2}}r \, \partial_\lambda Q_{\YM, \la} (r)=\frac{4\lambda\tanh^2(r/2)\sinh^{1/2}r}{(1+\la^2\tanh^2(r/2))^2}.
} 
Note that $\eta_0^{\la}$ is a positive solution to $\LL_{W_\la}\eta_0=0$ and we have $ \eta^\la_0 \in L^2([0, c))$ for all $c >0$.  

\begin{proof}[Proof of Proposition \ref{no neg spec}] The existence of $\z_0^{\la}$ for $\Lline$ and  $\eta_0^{\la}$ for $\LL_{W_\la}$ exclude the possibility of an eigenvalue with $\mu^2 = 0$. 
To prove the first statement, it thus suffices to rule out eigenvalues in $(-\infty, 0)$.
Suppose  such an eigenvalue exists. Then, as in the proof of Theorems \ref{S2 no eval} and \ref{YM no eval}, there exists $\mu \in \C$ with $\mu^{2} \le 0$ and an $L^{2}$ solution $\phi_{\mu}$ to~\eqref{phi m} that is strictly positive. Proceeding as in the proof of Theorem~\ref{S2 no eval} (resp. Theorem~\ref{YM no eval}) with $\beta_0=\zeta_{0}^\la$, (resp. $\be_0  =\eta_0^\la$) in place of $g(Q)$, for any $R > 0$ we obtain
\begin{equation*}
	\mu^{2} \int_{0}^{R} \beta_0(r) \phi_{\mu}(r) \, \ud r = - \phi'_{\mu}(R) \beta_{0}(R) + \phi_{\mu}(R) {\beta_{0}}'(R).
\end{equation*}
Arguing  as in  the proof Theorems \ref{S2 no eval} or  \ref{YM no eval}, we see that the left-hand side is strictly negative and decreasing in $R$. On the other hand, the right-hand side is non-negative for sufficiently large $R$, which is a contradiction.

The second statement follows from the fact that if $\mu^{2} \geq 1/4$ then there cannot exist any non-zero solution to $\calL_{\VV} \phi = \mu^{2} \phi$ in $L^{2}([1, \infty))$. To prove this statement, note that $\calL_{\VV} - \mu^{2}$ is well-approximated by $-\rd_{rr} - (\mu^{2} - 1/4)$, near $r = \infty$. Moreover, observe  that a fundamental system for $-\rd_{rr} f - (\mu^{2} - 1/4) f = 0$ is $\set{e^{\pm i \sqrt{\mu^{2} - 1/4} \, r}}$ when $\mu^{2} > 1/4$ and $\set{1, r}$ when $\mu^{2} = 1/4$, neither of which decay as $r \to \infty$. \qedhere
\end{proof}

\section{Existence of Gap Eigenvalues} \label{s:eval} 

In this subsection we carry out the proof of Theorem~\ref{e val}, which  in the end, comes down to an  elementary argument based on Sturm oscillation theory. We restrict to the case $k \ge 2$ since the corotational problem, i.e., $k=1$, has already been addressed in~\cite{LOS1}. 

Before beginning the proof, we first give a bit of intuition as to~\emph{why} $H_{V_{\la, k}}$ should have gap eigenvalues, at least in the large $\la$ regime. In short, the existence of 
gap eigenvalues can be attributed to the presence of an eigenvalue for the underlying scale invariant Euclidean problem at the zero energy threshold.  This is best understood in a rescaled setting. We introduce the notation  $$\rho:=  \la r/ 2, $$ and we refer to $\rho$ as the \emph{renormalized} coordinate. Setting $ \ti{Q}_{\la, k}( \rho) := Q_{\la, k}(r)$ we have 
\EQ{
\ti Q_{\la, k}( \rho) = 2 \arctan( \la^k \tanh^k( \rho/ \la))
} 
And note that  the map $\ti Q_{\la, k}$ can be well approximated by $Q_{\euc, k}( \rho) = 2 \arctan( \rho^k)$ on a $\rho$-interval of size  $\e \la$. It is basically  this phenomenon which we would like to exploit at the level of the spectral theory for the operators obtained by linearization about $\ti Q_{\la, k}$ and $Q_{\euc, k}$. Indeed, define  
\EQ{ \label{L def}
\rnL_{\lmb,k}  := -\rd_{\rho \rho} +  (k^2 - \frac{1}{4}) \frac{4}{ \la^2 \sinh^2 (2\rho/\la)} + \frac{1}{ \lmb^{2}} + \frac{4}{ \la^2}V_{\la,k}(2\rho/ \la).
}
We note that  $\rnL_{\lmb,k}$ is related to $\calL_{V_{{\lmb,k}}}$ as follows: Given a function $\phi(r)$ on $(0, \infty)$, define $  \phi_\la(\rho) := \phi(2\rho / \lmb)$. Then
\begin{equation*}
	\left(\rnL_{\lmb,k} \phi_\lambda\right)(\rho)
		= \frac{4}{\lmb^{2}} (\calL_{V_{\lmb,k}} \phi)(2\rho / \lmb)
		= \frac{4}{\lmb^{2}} [\calL_{V_{\lmb,k}} \phi]_\lambda(\rho).
\end{equation*}
This means that $\phi$ solves $\Lline \phi = \mu^2 \phi$ if and only if $\ti \phi(\rho) = \phi(2 \rho/ \la)$ solves $$\Lv \ti \phi =  \frac{4 \mu^2 }{ \la^2} \ti  \phi. $$ In the limit $\lambda\rightarrow\infty,~\rnL_{\lambda,k}$ formally tends to the operator 
\EQ{\label{re euc}
\LL_{\euc} \fy  :=& - \fy_{\rho \rho} + (k^2-\frac{1}{4}) \frac{1}{\rho^2}  \fy + V_{\euc, k}(\rho) \fy, \\
V_{\euc, k}(\rho):=& - k^2\frac{\rho^{2k-2}}{(1+ \rho^{2k})^2}.
}
The equation $\LL_{\euc} \varphi = 0$ possesses an explicit solution
\begin{equation} \label{eq:eucRes}
	\varphi_{\euc}(\rho) := \frac{\rho^{k+\frac{1}{2}}}{1 + \rho^{2k}} \in L^2([0, \infty)) \mfor k \ge 2.
\end{equation}

The Schr\"odinger operator $\LL_{\euc}$ arises by linearizing the $k$-equivariant Euclidean wave maps equation~\eqref{euc wm} around the ground state harmonic map $Q_{\euc, k}$. The explicit solution $\varphi_{\euc}$ is obtained from the scaling invariance of the problem and is an \emph{eigenfunction} as long as $k  \ge 2$. We note that $ \fy_{\euc}$ is a zero energy resonance in the case $k=1$;  
see \cite{KST} for more details.

In the $1$-equivariant case considered in~\cite{LOS1},  the formal limit $\Lv \to \LL_{\euc}$ was exploited by way of the following~\emph{renormalization}. Given a solution  $\phi_0$ to $$\Lv \phi_0 =  \frac{1}{\la^2} \phi_0,  \quad \phi_0 \in L^2([0, c)),  \, \, c>0$$ define its renormalization $\ti f( \rho)$ by 
\ant{
\ti f( \rho) := \frac{ \phi_0( \rho)}{ \fy_{\euc}( \rho)}
}
Then $\ti f $ solves 
\ant{
\left( \ti f' \fy_{\euc}^2 \right)'  =  \fy_{\euc}^2  \left( \Lv - \frac{1}{\la^2} - \LL_{\euc}\right)\ti f
}
To prove that $\Lv$ has an eigenvalue $ \mu^2 < \frac{1}{ \la^2}$  is suffices, by standard Sturm oscillation theory, to show that any solution $\phi_0 \in L^2([0, c))$ as above must change signs at least once.  After renormalization above, it then suffices that $\ti f$ must change signs. The fact that   $\Lv - \frac{1}{\la^2} - \LL_{\euc} \to 0 \mas \la \to \infty$ for each $\rho >0$ suggests that one can obtain good control over  $\ti f$ on an interval of size, $ 0 \le \rho \lesssim 1$. Also, the conclusions of Lemma~\ref{LL ev 1} in the $k =1$ case, give good control of $ \phi_0$ near $\rho = \infty$. The difficulty arises when trying to make these two regions, where $\ti f$ is controlled,  overlap. In~\cite{LOS1} this is achieved by extending the control given by the renormalization near $ \rho =0$ to an interval of size $ 0 \le \rho \lesssim \la$ which can be made arbitrarily large. The key is a certain a priori estimate for $  \ti f$ which holds under the contradiction hypothesis that $\ti f$ remains positive on the interval $[0, \la]$, see \cite[Lemma~$3.9$]{LOS1}. This  approach requires precise control over $\Lv - \frac{1}{\la^2} - \LL_{\euc}$ which is a complicated singular expression. 

We follow a somewhat different approach in this paper to deal with the higher equivariance classes $k \ge 2$. In fact, one may expect that it is easier to detect gap eigenvalues for the higher equivariance classes since the solution $\phi_{\euc}$ is an \emph{eigenvalue} for the $\LL_{\euc}$ when  $k  \ge 2$, rather than a \emph{resonance} when $k =1$. The presence of eigenvalues for an operator $-\p_{rr}  + V$ is a stable phenomenon under small changes to the potential $V$, while the presence of a resonance is extremely unstable. Of course the difference between $\LL_{\euc}$ and $\Lv - \frac{1}{\la^2}$ is not given by a small perturbation in any reasonable sense, and the previous statement is meant only as a rough heuristic. On the other hand, in the argument presented below the existence of gap eigenvalues is deduced by way of a polynomial divergence in $\la$ with power  $\la^{2k-2}$, rather than the delicate $\log(\la)$ divergence detected in the $1$-equivariant case in \cite{LOS1}; compare~\cite[Proof of Lemm~$3.9$]{LOS1} with the proof of Theorem~\ref{e val} below.

Rather than  confront the difference  $ \Lv - \frac{1}{\la^2} - \LL_{\euc}$ directly as in~\cite{LOS1} by renormalizing with respect to $\fy_{\euc}$, here we renormalize relative to $\z_0^\la$, which is defined in~\eqref{zdef} and is the unique $L^2_{\loc}$ solution to $\Lline \zeta = 0$. Recall that $\z_0^\la$  is obtained by differentiating $Q_{\la, k}( \rho)$ with respect to~$\la$. 
Since we will work exclusively in the renormalized coordinate $\rho =  \la r/ 2$ from now on, we will slightly abuse notation by writing  $\zeta_0^\la( \rho) =  \zeta_0^\la(  2 \rho/ \la)$. Then $\z_0^\la( \rho)$ is given by 
\EQ{ \label{zdef1} 
\zeta_0^\la( \rho):=\frac{2k\lambda^{k-1}\tanh^{k}( \rho/ \la)\sinh^{1/2}(2  \rho/ \la)}{1+\lambda^{2k}\tanh^{2k}( \rho/ \la)}
 }
 and solves 
 \EQ{ \label{Lz}
   \Lv \z_0^\la( \rho) = 0, \quad  \z_0^\la \in L^2([0, c]), \, \, c >0
   }
We note that $\zeta_0^\la$ has the same behavior near $ \rho =0$ as $\fy_{\euc}$. However, $\z_0^\la$ is not an eigenvalue for $\Lv$,  in fact $\z_0^\la(  \rho)$ grows like $e^{\rho/2 \la}$. This last point will not matter in the argument though. In fact $\zeta_0^\la$ is well-approximated by $\fy_{\euc}$ for $\rho \leq \eps \la$, where $\eps > 0$ is sufficiently small. In view of the decay of $\fy_{\euc}$ as $\rho\to\infty$ this means that $\zeta_0^\la$ decays at $\rho = \eps \la$ for large $\la$. Even though $\zeta_0^\la$ does not change sign, this decay is enough to show that $\phi_0,$ which satisfies an equation with a more attractive potential compared to $\zeta_0^\la,$ must change sign.

Given any solution $\phi$ to $\Lv \phi =  \frac{4\mu^2}{\la^2} \phi$ we define its renormalization with respect to $\z_0^\la$ (which we will now denote simply by $\z := \z_0^{\la}$) by 
\EQ{ \label{ren}
 f_{\mu}(\rho) =  \frac{\phi( \rho)}{\z( \rho)}
 }
 Note that by Lemma~\ref{LL ev 0} we can multiply  $\phi$ by a constant  to ensure that we have $(f_\mu(0), f'_{\mu}(0)) = (1, 0)$. 
 Then $f(\rho)$ solves 
 \EQ{\label{feq}
 ( f_\mu' \z^2)' ( \rho)= -\frac{ 4\mu^2}{\la^2} \,  \z^2( \rho)  f_\mu( \rho), \quad f_\mu(0) = 1,  \,\, f'_{\mu}(0)= 0
 }
 
 We are now ready to prove Theorem~\ref{e val}. 
 
 \begin{proof}[Proof of  Theorem~\ref{e val}$(i)$]
 We begin by establishing the existence of gap eigenvalues in the case that $k \ge 2$ is fixed and $\la$ is large enough.
 
 Let $\phi$ be a solution to 
 \EQ{ \label{phidef}
 \Lv \phi =   \frac{1}{\la^2}  \phi , \quad  \phi \in L^2([0, c]), \, \,  \forall c>0
 }
 and we multiply by a suitable constant so that its renormalization $f_{\frac{1}{2}}(\rho) = f ( \rho) =  \phi( \rho) / \z( \rho)$ as in~\eqref{ren} satisfies $f(0) = 1$. By~\eqref{feq} we see that for any $\rho_1 \ge 0$  
  \EQ{ \label{form}
  &f( \rho) = f( \rho_1) + f'( \rho_1) \z^2( \rho_1) \int_{\rho_1}^{\rho} \z^{-2}( \tau) \, d \tau - \frac{1}{\la^2} \int_{\rho_1}^{\rho}  \int_{\rho_1}^{\tau} \frac{  \z^2( \s)}{ \z^{2}( \tau)}  f( \s) \, d \s \, d \tau \\
  &f'( \rho)= f'( \rho_1) \z^2( \rho_1) \z^{-2}( \rho) -  \frac{1}{\la^2}\int_{\rho_1}^{\rho}  \frac{  \z^2( \s)}{ \z^{2}( \rho)}   f( \s) \, d \s 
  }
  By Sturm oscillation theory, it suffices to show that any such solution $\phi$ must change signs. We assume for contradiction that $\phi$ as in~\eqref{phidef} satisfies $\phi( \rho) \ge 0$ for all $ \rho  \ge 0$.  Since $\z( \rho) \ge 0$ for all $ \rho \ge 0$, this means that we are assuming for the sake of contradiction that 
  \ant{
  f( \rho) \ge 0 , \quad \forall \, \,  \rho \ge 0
  }
  By~\eqref{form} it follows that $f'( \rho)<0$ for all $ \rho > 0$,  which means that $f$ is strictly decreasing. Under these hypotheses, we prove the following claim. 
\begin{claim} \label{c:12}
  Set $ \rho_0 := \rho_0( \la)  :=  \la \arctanh( 1/ \la)$. If $k \ge 2$ is fixed, then there exists $\La_1 =  \La_1(k)$ large enough so that for $\la\geq\Lambda_1$
  \EQ{ \label{f12}
  f( \rho) \ge  \frac{1}{2}, \quad  \mfor  \rho \in [0, \rho_0( \la)].
  }
  \end{claim}
  \begin{proof}[Proof of Claim~\ref{c:12}]
  Setting $\rho_1 = 0$ in~\eqref{form} gives, 
  \EQ{ \label{f0}
  f( \rho)  = 1 - \frac{1}{\la^2} \int_{0}^{\rho} \z^{-2}( \tau) \int_{0}^{\tau}   \z^2( \s)   f( \s) \, d \s \, d \tau
  }
  Since $f( \rho)$ is decreasing, we have $f( \rho) \le 1$ and thus for all $ \tau \le \rho_0( \la)$, and using the explicit formula~\eqref{zdef1},
  \ant{
   \int_{0}^{\tau}   \z^2( \s)   f( \s) \, d \s  &\le \int_{0}^{\tau}   \z^2( \s)   \, d \s   =  \int_0^{\tau}\frac{4k^2\lambda^{2k-2}\tanh^{2k}( \s/ \la)\sinh( 2\s/ \la)}{(1+\lambda^{2k}\tanh^{2k}( \s/ \la))^2} \, d \s \\
   & \le 4k^2\lambda^{2k-2}\sinh( 2\tau/ \la) \int_0^{\tau}\tanh^{2k}( \s/ \la) \, d  \s \\
   & \le 4k^2\lambda^{2k-2}\sinh( 2\tau/ \la) \cosh^2(  \rho_0( \la)/ \la) \int_0^{\tau}\frac{\tanh^{2k}( \s/ \la)}{\cosh^{2}(  \s/ \la)} \, d \s
   }
   Using the definition of $\rho_0( \la)$ we have  $ \cosh^2(  \rho_0( \la)/ \la) =  \frac{\la^2}{ \la^2 -1}$. Since $ \frac{d}{d x}  \tanh^m x  =  m \tanh^{m-1}x  \cosh^{-2}{x}$ the last line above is equal to 
   \ant{
     &=  \frac{4k^2}{2k+1} \la^{2k-1}\frac{\la^2}{ \la^2 -1} \sinh( 2 \tau/ \la)  \int_0^{\tau} \frac{d}{d \s} ( \tanh^{2k+1}( \s/ \la) )\, d \s  \\
     & =  \frac{4k^2}{2k+1} \la^{2k-1}\frac{\la^2}{ \la^2 -1} \sinh( 2 \tau/ \la)  \tanh^{2k+1}( \tau/ \la)
     }
     Therefore, using the above as well as the definition of $\z$ in~\eqref{zdef1}, we have for all $ \rho \le \rho_0$
 \EQ{ \label{int12}
  \frac{1}{\la^2}  & \int_{0}^{\rho} \z^{-2}( \tau) \int_{0}^{\tau}   \z^2( \s)   f( \s) \, d \s \, d \tau  \\
   & \le   \frac{4k^2}{2k+1} \la^{2k-3}\frac{\la^2}{ \la^2 -1} \int_{0}^{\rho} \frac{\sinh( 2 \tau/ \la)  \tanh^{2k+1}( \tau/ \la)}{ \z^2( \tau)} \, d \tau \\
   & =   \la^{-1} \frac{\la^2}{ \la^2 -1} \frac{1}{2k+1} \int_0^{\rho} \tanh( \tau/ \la)( 1 +  \la^{2k} \tanh^{2k}(  \tau/ \la))^2 \, d \tau \\
   & \le \la^{-1} \frac{\la^2}{ \la^2 -1} \frac{1}{2k+1}\tanh( \rho_0/ \la)( 1 +  \la^{2k} \tanh^{2k}(  \rho_0/ \la))^2   \rho_0 \\
   &  = \la^{-1} \frac{\la^2}{ \la^2 -1} \frac{4}{2k+1} \arctanh( \la^{-1})  \le  \frac{1}{2} 
   }
   where the final inequality in the last line above holds 
    if $k \ge 2$ is fixed by taking $\la$ large enough. Inserting~\eqref{int12} into~\eqref{f0} yields the claim. 
  \end{proof}
  The lower bound on $f ( \rho)$ in Claim~\ref{c:12} allows us to find a strictly negative upper bound for $f'( \rho_0( \la))$. In particular we prove the following claim. 
  \begin{claim} \label{c:f'}
 Let $\rho_0 = \la \arctanh(1/ \la)$.  If $f( \rho) \ge \frac{1}{2}$ for all $ \rho \in [0, \rho_0]$, then 
 \EQ{ \label{f'0}
  \abs{f'( \rho_0)} \ge  \frac{ 1}{ \z^2( \rho_0)} \frac{k^2}{2k+2}  \la^{-5}
  }
  \end{claim}  
  
  \begin{proof}[Proof of Claim~\ref{c:f'}] 
  Setting $ \rho_1 = 0$  in~\eqref{form} and using the hypothesis $f( \tau) \ge \frac{1}{2}$  for $ \tau \in [0, \rho_0]$ gives, 
  \EQ{ \label{f'est}
   \abs{f'( \rho_0)}  &=   \abs{\frac{1}{\la^2}\int_{0}^{\rho_0} \frac{ \z^2( \s)}{ \z^2( \rho_0)}  f( \s) \, d \s }  \ge \frac{1}{2 \la^2} \frac{1}{\z^2( \rho_0)}  \int_0^{\rho_0} \z^2( \s) \, d \s
   }
  Using the formula~\eqref{zdef1} and the fact that for $\s \le \rho_0$ we have $1+ \la^{2k} \tanh^{2k}( \s/ \la) \le 1+ \la^{2k} \tanh^{2k}( \rho_0/ \la)  = 2$, we obtain
  \ant{
  \int_0^{\rho_0}& \z^2( \s) \, d \s =  \int_0^{\rho_0}\frac{4k^2\lambda^{2k-2}\tanh^{2k}( \s/ \la)\sinh( 2\s/ \la)}{(1+\lambda^{2k}\tanh^{2k}( \s/ \la))^2} \, d \s \\
  & \ge k^2 \lambda^{2k-2} \int_0^{\rho_0}\tanh^{2k}( \s/ \la)\sinh( 2\s/ \la) \, d \s   \\
  & = 2k^2 \lambda^{2k-2} \int_0^{\rho_0}\tanh^{2k+1}( \s/ \la) \cosh^2( \s/ \la) \, d \s \\
  & \ge  \frac{2k^2}{2k+2} \la^{2k-1} \int_0^{\rho_0}   \frac{d}{d \s} \tanh^{2k+2}( \s/ \la) \, d\s  =  \frac{2k^2}{2k+2} \la^{2k-1} \tanh^{2k+2}( \rho_0/ \la) \\
  &=  \frac{2k^2}{2k+2} \la^{-3}
  }
  Plugging the above into~\eqref{f'est} yields~\eqref{f'0}. 
  \end{proof} 
  Now, using Claim~\ref{c:12}, Claim~\ref{c:f'}, and setting $\rho_1 = \rho_0 = \la \arctanh(1/ \la)$ in~\eqref{form} we have, for any $ \rho \ge \rho_0$ 
  \EQ{ \label{frho}
  f( \rho)  &= f( \rho_0) + f'( \rho_0) \z^2( \rho_0) \int_{\rho_0}^{\rho} \z^{-2}( \tau) \, d \tau - \frac{1}{\la^2} \int_{\rho_0}^{\rho}  \int_{\rho_0}^{\tau} \frac{  \z^2( \s)}{ \z^{2}( \tau)}  f( \s) \, d \s \, d \tau \\
 &  \leq 1 + f'( \rho_0) \z^2( \rho_0) \int_{\rho_0}^{\rho} \z^{-2}( \tau) \, d \tau  \leq  1-   \frac{k^2}{2k+2}  \la^{-5} \int_{\rho_0}^{\rho} \z^{-2}( \tau) \, d \tau 
}
 Next, using again the explicit formula~\eqref{zdef1} 
 \ant{
 \int_{\rho_0}^{\rho}& \z^{-2}( \tau) \, d \tau =   \int_{\rho_0}^{\rho} \frac{(1+ \la^{2k} \tanh^{2k}(  \tau/ \la))^2}{ 4k^2 \la^{2k-2} \tanh^{2k}( \tau/ \la) \sinh(2 \tau/ \la)} \, d \tau \\
 & \ge \frac{\la^{2k+2}}{4k^2}  \int_{\rho_0}^{\rho}  \frac{ \tanh^{2k}( \tau/ \la)}{ \sinh (2 \tau/ \la) } \, d \tau 
  \,  \, = \,  \frac{\la^{2k+2}}{8k^2}  \int_{\rho_0}^{\rho}  \tanh^{2k-1}( \tau/ \la) \frac{1}{ \cosh^2( \tau/ \la)} \, d \tau \\
 & =   \frac{\la^{2k+3}}{16k^3}  \int_{\rho_0}^{\rho}   \frac{d}{d \tau} \tanh^{2k}( \tau / \la) \, d \tau  \, \,   =  \, \frac{\la^{2k+3}}{16k^3} \left( \tanh^{2k}( \rho/ \la) - \la^{-2k} \right)
  }
  Inserting the above into~\eqref{frho} gives 
  \ant{
  f( \rho) \le 1 - \frac{\la^{2k-2}}{32k(k+1)} \left( \tanh^{2k}( \rho/ \la) - \la^{-2k} \right)
  }
  Letting $ \rho \to \infty$ above then yields 
  \EQ{ \label{fneg}
   \lim_{\rho \to \infty} f( \rho) \le 1 - \frac{\la^{2k-2}}{32k(k+1)} \left( 1 - \la^{-2k} \right)
  }
  
  If $k \ge 2$ is fixed, we can find $\La_0(k) \ge \La_1(k)$ large enough so that for all $\la \ge \La_0$, the right-hand side of~\eqref{fneg} is negative, yielding a contradiction -- here $\La_1(k)$ is as in Claim~\ref{c:12}. This proves the existence of gap eigenvalues for $\la \ge \La_0$.  
  
 
 \vspace{\baselineskip}
 Next we prove that for fixed $k \ge 2$ and $\lambda$ large enough,  the eigenvalues we have found are simple and unique. We show that any eigenfunction $\psi_\mu\in L^2$ solving $\LL_{V_{\lambda,k}}\psi_\mu=\mu^2\psi_\mu,$ for $\mu^2 \in (0, \frac{1}{4})$ cannot change signs as long as $\la$ is large enough. 
 
 First we show that  for an appropriately chosen constant $C,$ and $\lambda$ large, $\psi_\mu$ does not change sign in the interval $[\frac{C}{\lambda},\infty).$ Without loss of generality, we may assume that $\psi_\mu(r)$ is positive for large $r.$ Define $m > 0$ by $m^{2} = \frac{1}{4} - \mu^{2}$. We  compare $\psi_\mu$ with $h(r):=e^{-m r}$, which up to scaling is the unique nonzero $L^2$ solution of $\partial_{rr} h=m^2 h$.  After suitable renormalization, it is clear from~\eqref{LLdef} and the exponential decay of $V_{\la, k}(r)$ as $ r \to \infty$, that  we may assume that $\psi_\mu (r)=  e^{-m r} + o(e^{-m r})$ as $r\rightarrow\infty.$ Defining
\[\WW(r) : = W[\psi_\mu,h](r)=\psi_\mu(r) h^\prime(r)-\psi_\mu^\prime(r) h(r),\]
we have
\begin{equation*}
	\WW'(r)
	= - \left(\frac{k^2\cos(2\Qu)-\frac{1}{4}}{ \sinh^{2} r}  \right) \psi_\mu(r) h(r).
\end{equation*}
It follows from the definition of $Q_{\la, k}$, which is explicit,  that if $\lambda$ is large enough, there exists a constant $C>0$ so that for $r\geq C/\lambda$
\ali{\label{new potential bound}
\frac{k^2\cos(2\Qu)-\frac{1}{4}}{ \sinh^{2} r}>0,
}
and therefore $\WW^\prime(r) \leq 0,$ so long as $\psi_\mu$ is positive (note that $h > 0$ everywhere). Assuming~\eqref{new potential bound} for the moment, let $R$ denote the largest zero of $\psi_\mu$ and for contradiction assume $R\geq C/\lambda$. Then $\WW^\prime(r)<0$ and $\psi_\mu \sim e^{-m r}$ as $r\rightarrow\infty$ imply that $\WW(R) \geq 0$. This means that
\[\lim_{r\rightarrow R^{+}} \frac{h^\prime(r)}{h(r)} \geq\lim_{r\rightarrow R^{+}}\frac{\psi_\mu^\prime(r)}{\psi_\mu(r)}=\infty,\]
and therefore we must have $h(R)=0$, which is impossible. In the case of a threshold resonance, simply run the the same argument as above comparing with $h   \equiv 1$, which up to scaling is the unique nonzero bounded solution to $h_{rr} = 0$ -- we omit the details here since the argument is very similar. 


 It remains to prove that $\psi_\mu$ cannot change signs on the interval $[0, C/ \la]$ for large enough $\lambda$ (again the same argument works here in the case of a threshold resonance).   Defining the renormalization $f_\mu$ as in~\eqref{ren} by
\ant{
&f_\mu(\rho)=\frac{\psi_\mu(2\rho/\lambda)}{\zeta(\rho)}, \quad f_{\mu}(0) = 1, \, \, f'_\mu(0) = 0 \\
&f_{\mu}( \rho) = 1 - \frac{4 \mu^2}{\la^2} \int_{0}^{\rho}  \int_{0}^{\tau} \frac{  \z^2( \s)}{ \z^{2}( \tau)}  f( \s) \, d \s \, d \tau
}
we note that it suffices to show that $f_{\mu}$ cannot change signs in the interval $[0, \frac{C}{2}]$. Using the explicit formula for $\z$ we see that 
\ant{
\abs{ \frac{4 \mu^2}{\la^2} \int_{0}^{C/2}  \int_{0}^{\tau} \frac{  \z^2( \s)}{ \z^{2}( \tau)}   \, d \s \, d \tau}     = O( \la^{-2}) \to 0 \mas \la \to \infty
}
and thus it follows for a Volterra-type iteration argument (see for example~\cite[Proof of  Claim $3.12$]{LOS1} that 
\ant{
\sup_{ \rho \in [0, C/2]} \abs{ f_{\mu}( \rho) - 1} = o(1) \mas \la \to \infty
}
we conclude that $f_\mu$ is positive in $[0, \frac{C}{2}]$ as long as $\la$ is large enough,  as desired. 
  
  \vspace{\baselineskip}
  
  Finally, to complete the proof of Theorem~\ref{e val}$(i)$ by showing that the simple gap eigenvalue $\mu_{\la}^2$ migrates to $0$ as $\la \to \infty$, i.e., we prove~\eqref{mu to 0}. By Sturm oscillation and the definition of $\Lv$,  it suffices to prove the following claim. 
  
  \begin{claim} \label{c:mig} Let $\ba \mu^2 \in (0, 1/4]$. Then, for $\la$ large enough (depending on $\ba \mu^2$), the solution $\phi_0$ to the ODE
  \ant{
   \Lv \phi_0 = \frac{4\ba \mu^2}{ \la^2} \phi_0,  \quad  \phi_0( \rho)/ \z (\rho) \to 1 \mas \rho \to 0
   }
   must change sign. 
  \end{claim}
  
  \begin{proof}[Proof of Claim~\ref{c:mig}] Defining the renormalization 
  \ant{
  f_{\ba \mu}( \rho)  =  \phi_0( \rho)/ \z (\rho)
  }
  we have, as in~\eqref{form} that 
   \ant{
  &f_{\ba  \mu}( \rho) = f_{\ba  \mu}( \rho_1) + f'_{\ba  \mu}( \rho_1) \z^2( \rho_1) \int_{\rho_1}^{\rho} \z^{-2}( \tau) \, d \tau - \frac{4 \ba\mu^2}{\la^2} \int_{\rho_1}^{\rho}  \int_{\rho_1}^{\tau} \frac{  \z^2( \s)}{ \z^{2}( \tau)}  f_{\ba  \mu}( \s) \, d \s \, d \tau \\
  &f'_{\ba \mu}( \rho)= f'_{\ba  \mu}( \rho_1) \z^2( \rho_1) \z^{-2}( \rho) -  \frac{4 \ba\mu^2}{\la^2}\int_{\rho_1}^{\rho}  \frac{  \z^2( \s)}{ \z^{2}( \rho)}   f_{\ba  \mu}( \s) \, d \s 
  }
  for any $\rho_1$ fixed. The proof of Claim~\ref{c:mig} then follows from the exact same argument used to prove that $f( \rho)$ as in~\eqref{form} changes signs for $\la$ large enough -- note that the only difference between the above and~\eqref{form} is the factor of $4 \ba \mu^2$ in front of the integrals. 
  \end{proof}
  
 This completes the proof of Theorem~\ref{e val}$(i)$.
 \end{proof}

We turn next to the proof of Theorem \ref{e val}$(ii)$. The proof still relies on the renormalization technique introduced above, but instead of considering the formal limit of $\LL_{\la,k}$ as $\la\to\infty$ we need to consider the limit $k\to\infty$. As in the case of large $\la,$ the necessary information for the proof of existence of gap eigenvalues is encoded in the solution $$\zeta_0^\la(r)=\frac{2k\la^{k-1}\tanh^k(r/2)\sinh^{1/2}r}{1+\la^{2k}\tanh^{2k}(r/2)}.$$
As usual we let $\phi$ be the solution of $\LL_{\la,k}\phi=\frac{1}{4}\phi$ with the same asymptotic behavior as $\zeta_0^\la$ near $r=0$ and introduce $$f(r)=\frac{\phi(r)}{\zeta_0^\la(r)},\qquad f(0)=1,~f^\prime(0)=0.$$ For simplicity of notation we will henceforth write $\zeta$ instead of $\zeta_0^\la.$ The integral equation satisfied by $f$ is (here $r_0 \ge0$ is an arbitrary constant)
\ant{
&f(r)=f(r_0)+\zeta^2(r_0)f^\prime(r_0)\int_{r_0}^r\frac{dt}{\zeta^2(t)}-\frac{1}{4}\int_{r_0}^r\int_{r_0}^t\frac{\zeta^2(s)}{\zeta^2(t)}f(s)dsdt,\\
&f^\prime(r)=\frac{\zeta^2(r_0)f^\prime(r_0)}{\zeta^2(r)}-\frac{1}{4\zeta^2(r)}\int_{r_0}^r\zeta^2(t)f(t)dt.
}
In order to analyze this integral uniformly in $k$ as $k\to\infty,$ we again introduce  the \emph{renormalized variable} $$\rrho=\lambda^k\tanh^k(r/2)=\Theta\tanh^k(r/2).$$ Recall from the proof of Theorem~\ref{S2 no eval}$(ii)$ that $\rrho\leq\Theta$ and
$$dr=\left[\left(\frac{k}{2}\right)\left(\left(\frac{\Theta}{\rrho}\right)^{\frac{1}{k}}-\left(\frac{\Theta}{\rrho}\right)^{-\frac{1}{k}}\right)\right]^{-1}\frac{d\rrho}{\rrho}$$
Abusing notation by writing $\zeta(\rrho)$ instead of $\zeta(r)$ we have
\ant{
\frac{\zeta^{2}(\ssigma)}{\zeta^{2}(\ttau)}=\frac{\left(\frac{k}{2}\right)\left(\left(\frac{\Theta}{\ttau}\right)^{\frac{1}{k}}-\left(\frac{\Theta}{\ttau}\right)^{-\frac{1}{k}}\right)}{\left(\frac{k}{2}\right)\left(\left(\frac{\Theta}{\ssigma}\right)^{\frac{1}{k}}-\left(\frac{\Theta}{\ssigma}\right)^{-\frac{1}{k}}\right)}\cdot\frac{\sigma^2(1+\tau^2)^2}{\tau^2(1+\sigma^2)^2},
}
and the integral equations for $f$ can be written as (where now $\prime:=\frac{d}{d\rrho}$)
\ali{\label{large k int eq}
&f(\rrho)=f(\rrho_0)+\rrho_0f^\prime(\rrho_0)\frac{\rrho_0^2}{(1+\rrho_0^2)^2}\int_{\rrho_0}^\rrho\frac{(1+\ttau^2)^2}{\ttau^2}\frac{d\ttau}{\ttau}\\
&\qquad\quad-\frac{1}{4}\int_{\rrho_0}^\rrho\int_{\rrho_0}^\ttau\frac{\ssigma^2(1+\ttau^2)^2}{\ttau^2(1+\ssigma^2)^2}\left[\left(\frac{k}{2}\right)\left(\left(\frac{\Theta}{\ssigma}\right)^{\frac{1}{k}}-\left(\frac{\Theta}{\ssigma}\right)^{-\frac{1}{k}}\right)\right]^{-2}f(\sigma)\frac{d\ssigma}{\ssigma}\frac{d\ttau}{\ttau},\\
&f^{\prime}(\rrho)=f^\prime(\rrho_0)-\frac{1}{4}\frac{(1+\rrho^2)^2}{\rrho^3}\int_{\rrho_0}^\rrho\frac{\ssigma^2}{(1+\ssigma^2)^2}\left[\left(\frac{k}{2}\right)\left(\left(\frac{\Theta}{\ssigma}\right)^{\frac{1}{k}}-\left(\frac{\Theta}{\ssigma}\right)^{-\frac{1}{k}}\right)\right]^{-1}\frac{d\ssigma}{\ssigma}.
}

\begin{proof}[Proof of Theorem \ref{e val}$(ii)$]
Assume for contradiction that $f(\rrho$) is everywhere non-negative. It follows from the expression \eqref{large k int eq} that $f$ is decreasing and $f(\rrho)\leq1$ for all $\rrho\leq\Theta.$ We begin by providing a lower bound on $f(\rrho)$ for $\rrho\in[0,A]$ where $A$ is to be determined. Using Lemma \ref{k limit} we have
\ant{
f(\rrho)&\geq1-\frac{1}{4}\int_{0}^\rrho\int_{0}^\ttau\frac{\ssigma^2(1+\ttau^2)^2}{\ttau^2(1+\ssigma^2)^2}\left[\left(\frac{k}{2}\right)\left(\left(\frac{\Theta}{\ssigma}\right)^{\frac{1}{k}}-\left(\frac{\Theta}{\ssigma}\right)^{-\frac{1}{k}}\right)\right]^{-2}\frac{d\ssigma}{\ssigma}\frac{d\ttau}{\ttau}\\
&\geq1-\frac{1}{4}\int_{0}^\rrho\int_{0}^\ttau\frac{\ssigma^2(1+\ttau^2)^2}{\ttau^2(1+\ssigma^2)^2\log^2\left(\frac{\Theta}{\ssigma}\right)}\frac{d\ssigma}{\ssigma}\frac{d\ttau}{\ttau}\\
&=1-\frac{1}{4}\int_0^\rrho \frac{\ssigma}{(1+\ssigma^2)^2\log^2\left(\frac{\Theta}{\ssigma}\right)}\int_\ssigma^\rrho\frac{(1+\ttau^2)^2}{\ttau^3}d\ttau d\ssigma\\
&\geq1-C(1+A^4)\int_0^\rrho\frac{1}{\ssigma\log^2\left(\frac{\Theta}{\ssigma}\right)}d\ssigma\geq1-\frac{C(1+A^4)}{\log \left(\frac{\Theta}{A}\right)},
} 
where to evaluate the last integral we have used the substitution $\ssigma^\prime=\log\left(\frac{\Theta}{\ssigma}\right).$ By taking $\Theta$ large compared to $A$ we can guarantee that the last quantity on the right hand side of the estimate above is bounded away from zero.

The next step in the proof consists of finding a lower bound on $|f^\prime(A)|.$ Using the expression \eqref{large k int eq}, the previous bound on $f(A),$ Lemma \ref{k limit}, and the dominated convergence theorem, we get
\ant{
f^\prime(A)&=-\frac{1}{4}\frac{(1+A^2)^2}{A^3}\int_0^Af(\ssigma)\frac{\ssigma^2}{(1+\ssigma^2)^2}\left[\left(\frac{k}{2}\right)\left(\left(\frac{\Theta}{\ssigma}\right)^{\frac{1}{k}}-\left(\frac{\Theta}{\ssigma}\right)^{-\frac{1}{k}}\right)\right]^{-2}\frac{d\ssigma}{\ssigma}\\
&\leq-\frac{C(1+A^2)}{A^3}\left(1-\frac{C(1+A^4)}{\log^2\left(\frac{\Theta}{A}\right)}\right)\left((1+A^2)^{-2}\int_0^A\frac{\ssigma}{\log^2\left(\frac{\Theta}{\ssigma}\right)}d\ssigma+o_k(1)\right)\\
&\leq-\frac{C(1+A^2)}{A^3}\left(1-\frac{C(1+A^4)}{\log^2\left(\frac{\Theta}{A}\right)}\right)\left((1+A^2)^{-2}\Theta^2\int_{\frac{A}{2\Theta}}^{\frac{A}{\Theta}}\frac{\ssigma^\prime d\ssigma^\prime}{\log^2\ssigma^\prime}+o_k(1)\right)\\
&\leq-\frac{C(1+A^2)}{A^3}\left(1-\frac{C(1+A^4)}{\log^2\left(\frac{\Theta}{A}\right)}\right)\left(\frac{A^2}{(1+A^2)^2\log^2\left(\frac{\Theta}{A}\right)}+o_k(1)\right).
}
Choosing $A=1$ and $\Theta$ large we conclude that $$f^\prime(1)\leq-C\left(1-\log^{-2}\Theta\right)\left(\log^{-2}\Theta+o_k(1)\right)\leq-C(\log^{-2}\Theta+o_k(1)).$$ Finally, going back to the expression \eqref{large k int eq} we see that
\ant{
f(\Theta/2)&\leq1-C(\log^{-2}\Theta+o_k(1))\int_1^{\Theta/2}\frac{(1+\ttau^2)^2}{\ttau^2}\frac{d\ttau}{\ttau}\\
&\leq1-C\Theta^2(\log^{-2}\Theta+o_k(1)).
}
Choosing $\Theta$ and $k$ large enough we conclude that $f(\Theta/2)$ is negative, which is a contradiction.
\end{proof}

The proof of Theorem \ref{YM e val} is very similar to the proof of Theorem~\ref{e val}$(i)$. We give a brief sketch. 
\begin{proof}[Proof of Theorem \ref{YM e val}]

As in the proof of Theorem~\ref{e val}, to prove the existence of a gap eigenvalue for large $\la$, we show that any solution $\phi_0^\la$ to 
\ant{
\LL_{W_\la} \phi_0^\la =  \frac{1}{4} \phi_0^\la, \quad \phi_0^\la \in L^2([0, c))  \mfor c>0 
}
must change signs. Here we renormalize with respect to $\eta_0^\la$ defined in~\eqref{edef}. Indeed passing to the renormalized coordinate $\rho:=  \la r/2$, we write $\eta( \rho):= \eta_0^{\la}(2 \rho/ \la)$ and $\phi( \rho) =  \phi_0^\la( 2 \rho/ \la)$ and set 
\EQ{
f( \rho):=  \frac{ \phi( \rho)}{ \eta( \rho)}.  
}
and we multiply by a suitable constant so that $(f(0), f'(0)) = (1, 0)$. We recall from~\eqref{edef} that 
\EQ{ \label{eta}
\eta( \rho)=  \frac{4 \la \tanh^2( \rho/ \la) \sinh^{\frac{1}{2}}(2 \rho/ \la)}{(1+ \la^2 \tanh^2( \rho/ \la))^2}
}
Then for any fixed $\rho_1 \ge 0$ we have 
 \ant{
  &f( \rho) = f( \rho_1) + f'( \rho_1) \eta^2( \rho_1) \int_{\rho_1}^{\rho} \eta^{-2}( \tau) \, d \tau - \frac{1}{\la^2} \int_{\rho_1}^{\rho}  \int_{\rho_1}^{\tau} \frac{  \eta^2( \s)}{ \eta^{2}( \tau)}  f( \s) \, d \s \, d \tau \\
  &f'( \rho)= f'( \rho_1) \eta^2( \rho_1) \eta^{-2}( \rho) -  \frac{1}{\la^2}\int_{\rho_1}^{\rho}  \frac{  \eta^2( \s)}{ \eta^{2}( \rho)}   f( \s) \, d \s 
  }
Assuming for contradiction that $f( \rho) \ge 0$ for all $ \rho \ge 0$, it follows that $f'$ is negative and hence $f$ is strictly decreasing. Arguing exactly as in Claim~\ref{c:12}, using the explicit formula for $\eta$,  one can choose $\la$ large enough so that $f( \rho) \ge \frac{1}{2}$ for all $\rho \in [0, \rho_0]$, where $\rho_0:= \rho_0( \la):= \la \arctanh(1/ \la)$. Then, proceeding as in the proof of Claim~\ref{c:f'}, one can use the fact that $f( \rho) \ge \frac{1}{2}$ on $[0, \rho_0]$ along with~\eqref{eta} to show that 
\ant{
 \abs{f'( \rho_0)} &\ge \frac{1}{2 \la^2} \frac{1}{ \eta^2( \rho_0)} \int_0^{\rho_0} \eta^2( \s) \, d \s\\
 & \ge \frac{1}{ 2\eta^2( \rho_0)} \int_0^{\rho_0} \tanh^4( \s/ \la)\sinh( \s/ \la) \cosh(\s/ \la) \, d \s \\
 & \ge c \la^{-5}  \frac{1}{ \eta^2( \rho_0)}
 }
It then follows from the integral equation for $f$ with $\rho_1 = \rho_0$ that 
\EQ{
f( \rho) \le 1 - c \la^{-5} \int_{\rho_0}^{\rho} \eta^{-2}( \tau) \, d \tau
}
Using again the explicit formula~\eqref{eta} we can deduce that 
\ant{
\lim_{\rho \to \infty} f( \rho) \le 1 - c \la^2( 1 -  \la^{-4})
}
which is negative for $\la$ large enough. This is a contradiction. Therefore $f$, and thus also $\phi_0$ must change signs. By Sturm oscillation it follows that there exists an eigenvalue in the spectral gap $(0, 1/4)$. 


For the uniqueness statement, arguing as in the proof of Theorem \ref{e val}, we can reduce matters to proving the analog of~\eqref{new potential bound} in the Yang-Mills setting. In particular,  it suffices to show that the following inequality holds for $r\geq C/\lambda$ for an appropriate constant $C$ and large $\lambda:$
\ant{
4(g_\YM g^\prime_\YM)^\prime(Q_{\YM, \la})-\frac{1}{4}\geq0.
}
Using the definition $g_\YM(v)=v-\frac{v^2}{2}$ this reduces to the estimate
\ant{
4(1-3Q_{\YM, \la} + \frac{3}{2} Q^2_{\YM, \la})- 1/4 \ge 0
}
which holds, say for 
\ant{
Q_{\YM, \la}(r)\geq  5/3. 
}
Using the definition of $Q_{\YM, \la}$, this is equivalent to 
\ant{
\tanh^2(r/2) \ge 5/ \la^2
}
for large $\lambda.$ The rest of the proof of the uniqueness statement follows from an identical argument as in the proof of Theorem~\ref{e val} and we omit the details. 

Finally, to prove that the unique simple eigenvalue $\mu_{\la}$ satisfies $\mu_{\la} \to 0$ as $\la \to \infty$, we argue exactly as in Claim~\ref{c:mig}. This completes the proof. 
\end{proof}

\bibliographystyle{plain}
\bibliography{researchbib}

 \bigskip

\centerline{\scshape Andrew Lawrie, Sung-Jin Oh}
\smallskip
{\footnotesize
 \centerline{Department of Mathematics, The University of California, Berkeley}
\centerline{970 Evans Hall \#3840, Berkeley, CA 94720, U.S.A.}
\centerline{\email{ alawrie@math.berkeley.edu, sjoh@math.berkeley.edu}}
} 

 \medskip

\centerline{\scshape Sohrab Shahshahani}
\medskip
{\footnotesize
 \centerline{Department of Mathematics, The University of Michigan}
\centerline{2074 East Hall, 530 Church Street
Ann Arbor, MI  48109-1043, U.S.A.}
\centerline{\email{shahshah@umich.edu}}
} 

\end{document}